\documentclass[11pt,a4paper,oneside]{article}
\usepackage[utf8]{inputenc}
\usepackage[T1]{fontenc}
\usepackage{amsmath,amssymb,amsthm}
\usepackage{mathrsfs}
\usepackage{listings}
\usepackage{enumerate}
\usepackage{wasysym}
\usepackage{graphicx}
\usepackage{caption}
\usepackage{braket}
\usepackage[dvipsnames]{xcolor}
\usepackage{tikz}
\usepackage{tcolorbox}
\usepackage{hyperref}
\usepackage{geometry}
\usepackage{bookmark}
\usepackage{pdfpages}
\usepackage{xcolor}

\usetikzlibrary{decorations.pathreplacing}
\usetikzlibrary{quotes,angles}

\tcbuselibrary{theorems}
\tcbuselibrary{breakable}
\tcbuselibrary{skins}
\hypersetup{
    colorlinks,
    citecolor=red,
    filecolor=black,
    linkcolor=blue,
    urlcolor=black
}

\usepackage[T1]{fontenc}
\tcbuselibrary{breakable}
\usepackage{titlesec, blindtext, color}
\definecolor{gray75}{gray}{0.75}
\newcommand{\hsp}{\hspace{20pt}}
\titleformat{\chapter}[hang]{\Huge\bfseries}{\thechapter\hsp\textcolor{gray75}{|}\hsp}{0pt}{\Huge\bfseries}

\newtcbtheorem{Lemma}{Lemma}%
{colback=green!5,colframe=green!35!black,fonttitle=\bfseries}{theorem}

\newtcbox{\mybox}[1]{colback=red!5!white,colframe=red!75!black,fonttitle=\bfseries,title=#1}

\newtcolorbox{ctheobox}{colback=green!10!white,colframe=green!10!white, sharp corners,breakable}

\newtcolorbox{clembox}{colback=ForestGreen!10!white,colframe=ForestGreen!10!white, sharp corners,breakable}

\newtcolorbox{cdbox}{colback=BrickRed!10!,colframe=BrickRed!10!white, sharp corners,breakable}

\newtcolorbox{ckorbox}{colback=Violet!10!,colframe=Violet!10!white, sharp corners,breakable}

\newtcolorbox{cbembox}{colback=BurntOrange!20!,colframe=BurntOrange!20!white, sharp corners,breakable}

\newtheorem{theorem}{{Theorem}}[section]
\newtheorem{lemma}[theorem]{{Lemma}}
\newtheorem{definition}[theorem]{{Definition}}
\newtheorem{korollar}[theorem]{{Corollary}}
\newtheorem{bemerkung}[theorem]{{Remark}}

\newcommand{\N}{\mathbb{N}}
\newcommand{\R}{\mathbb{R}}

\newcommand{\tr}{\operatorname{tr}}

\newcommand{\Sp}{\mathbb{S}}

\newcommand\blfootnote[1]{%
  \begingroup
  \renewcommand\thefootnote{}\footnote{#1}%
  \addtocounter{footnote}{-1}%
  \endgroup
}

\renewcommand{\theequation}{\arabic{section}.\arabic{equation}}
\begin{document}

\author{Jan-Henrik Metsch
\thanks{Corresponding author: J.-H. Metsch, Department of Mathematics, University of Freiburg, Germany ({\tt\small jan.metsch@math.uni-freiburg.de})}
}
\title{On the Conditional Existence of Foliations by CMC and\\
Willmore Type Half-Spheres}
\maketitle

\begin{abstract}
We study half-spheres with small radii $\lambda$ sitting on the boundary of a smooth bounded domain while meeting it orthogonally. Even though it is known that there exist families of CMC and Willmore type half-spheres near a nondegenerate critical point $p$ of the domains boundaries mean curvature, it is unknown in both cases whether these provide a foliation of any deleted neighborhood of $p$. We prove that this is not guaranteed and establish a criterion in terms of the boundaries geometry that ensures or prevents the respective surfaces from providing such a foliation. This perhaps surprising phenomenon of conditional foliations is absent in the closely related Riemannian setting, where a foliation is guaranteed. We show how this unconditional foliation arises from symmetry considerations and how these fail to apply to the `domain-setting'.
\end{abstract}


\vspace{-1cm}

\setcounter{equation}{0}
\section{Introduction}
The \emph{Area} and the \emph{Willmore energy} of a regular surface $f:\Sigma\rightarrow\R^3$ are defined respectively as 
$$A[f]:=\int_{\Sigma}d\mu_f\hspace{.5cm}\textrm{and}\hspace{.5cm}
\mathcal W[f]:=\frac14\int_{\Sigma}H[f]^2 d\mu_f.$$
Let $\Omega\subset\R^3$ be a bounded smooth domain and put $S:=\partial\Omega$. We consider the class
\begin{align*}
    \mathcal M(S)&:=\left\{\phi\in C^{4,\alpha}(\Sp^2_+,\R^3)\textrm{ immersed}\ \bigg|\ \phi(\partial \Sp^2_+)\subset S,\ \frac{\partial\phi}{\partial\eta}=N^S\circ\phi\right\}
\end{align*}
of immersed surfaces meeting $S$ orthogonally along the boundary. Here $\eta$ and $N^S$ denote the interior unit normals of $\Sp^2_+$ and $\Omega$ along their respective boundaries. We denote by $\mathcal M_\lambda(S)$ the subclass of surfaces with area $2\pi\lambda^2$. In their article \cite{AK}, Alessandroni and Kuwert study critical points of the Willmore energy inside the class $\mathcal M_\lambda$. To be concrete $\phi\in\mathcal M_\lambda(S)$ is critical, when 
\begin{equation}\label{problem}
\begin{aligned}
    &W[\phi]=\alpha H[\phi]\hspace{.5cm}&\textrm{in $\Sp^2_+$},\\
    &\frac{\partial H}{\partial \eta}+h^ S(\nu,\nu)H=0\hspace{.5cm}&\textrm{on $\partial\Sp^2_+$}.
\end{aligned}
\end{equation}
Here $W$ is the scalar Willmore operator, $H$ is the scalar mean curvature, $\nu$ is the normal of $\phi$ and $\alpha$ is a Lagrange multiplier. Surfaces satisfying Equation \eqref{problem} are referred to as surfaces of \emph{Willmore type}. Alessandroni and Kuwert construct critical surfaces $\phi\in\mathcal M_\lambda(S)$ for small enough $\lambda\leq\lambda_0(\Omega)$ and analyze the behavior of these in the limit $\lambda\rightarrow 0^ +$. They prove:
\begin{enumerate}[(1)]
    \item For $\lambda\rightarrow 0^ +$ critical surfaces can only concentrate at critical points of the domain's boundary's mean curvature $H^ S$.
    \item Given a nondegenerate critical point $p\in S$, of $H^ S$ there exist $\lambda_0>0$, a neighborhood $U$ of $p$, a smooth curve $\gamma:[0,\lambda_0)\rightarrow S$ satisfying $\gamma(0)=p$ and critical $\phi^ \lambda$ with \emph{Riemannian barycenter} -- a kind of center point on $S$ (see Appendix 1 in \cite{AK}) -- $C[\phi^ \lambda]$ such that 
    $$\begin{aligned}
        &\phi\in\mathcal M_\lambda(S)\textrm{ critical}\\
        &\hspace{.8cm}C[\phi]\in U
    \end{aligned}
    \hspace{.5cm}\Longleftrightarrow\hspace{.5cm}
    \phi=\phi_\lambda\textrm{ and }C[\phi]=\gamma(\lambda).
    $$
\end{enumerate}

A closely related problem is the study of CMC half-spheres. We put 
\begin{align*}
    \mathcal N(S)&:=\left\{\phi\in C^{2,\alpha}(\Sp^2_+,\R^3)\textrm{ immersed}\ \bigg|\ \phi(\partial \Sp^2_+)\subset S,\ \frac{\partial\phi}{\partial\eta}=N^S\circ\phi\right\}
\end{align*}
A surface $\phi\in \mathcal N(S)$ together with the boundary $S$ encloses a region, whose volume we denote by $V[\phi]$. We denote the subclass of all $\phi\in\mathcal N(S)$ that satisfy $V[\phi]=\frac23\pi\lambda^3$ by $\mathcal N_\lambda(S)$. In their article \cite{bellettinifusco} Bellettini and Fusco study critical points of the area $A$ inside the class $\mathcal N_\lambda(S)$ -- that is solutions to 
\begin{equation}\label{CMCEquationIntro}
H[\phi]=\alpha
\end{equation}
where $\alpha$ is a Lagrange multiplier. Surfaces satisfying Equation \eqref{CMCEquationIntro} are referred to as CMC-surfaces. Just as in the Willmore case, Bellettini and Fusco construct critical surfaces $\phi\in \mathcal N_\lambda(S)$ for small values $\lambda\leq\lambda_0(\Omega)$ and analyze their behavior in the limit $\lambda\rightarrow0^+$. Paraphrasing their result, they establish the analog statement of (2): Given  a nondegenerate critical point $p\in S$ of $H^ S$, there exist $\lambda_0>0$, a neighborhood $U$ of $p$, a smooth curve $\gamma:[0,\lambda_0)\rightarrow S$ satisfying $\gamma(0)=p$ and critical $\phi^ \lambda$ with \emph{barycenter} $C[\phi^ \lambda]=\gamma(\lambda)$ such that 
\begin{equation}\label{CMCSecondStatment}
\begin{aligned}
        &\phi\in\mathcal N_\lambda(S)\textrm{ critical}\\
        &\hspace{.8cm}C[\phi]\in U
    \end{aligned}
    \hspace{.5cm}\Longleftrightarrow\hspace{.5cm}
    \phi=\phi_\lambda\textrm{ and }C[\phi]=\gamma(\lambda).
\end{equation}

Both of these results, as well as the analysis to derive them, are similar to those obtained in the study of embedded CMC and Willmore type spheres in a Riemannian manifold $M$. For both the Willmore and the CMC case, the analog of Statement (1) holds when the domain's mean curvature $H^S$ is replaced by the ambient manifold's scalar curvature $\operatorname{Sc}$. In the Willmore case, this was first established by Laurain and Mondino in \cite{laurain2015concentration} and was verified by Lamm, Metzger and Schulze in \cite{lamm2009small}. In the CMC, the result is due to  Laurain \cite{LaurainConcentrationCMC}. The construction of area-constrained Willmore surfaces close to geodesic spheres originates in Mattuschka \cite{Mattuschka}. Here the analog of Statement (2), again with the replacement $H^S\rightarrow \operatorname{Sc}$, is derived. The analog of Statement \eqref{CMCSecondStatment} in the Riemannian setting, yet again with $\operatorname{Sc}$ instead $H^S$, is due to Ye \cite{ye}.\\

In both the CMC and the Willmore case, it is also known that the near a nondegenerate point $p$ of $\operatorname{Sc}$ the critical surfaces `close' to $p$ provide a foliation of a neighborhood $U$ of $p$. This sort of result was pioneered by Ye in \cite{ye}, where he considered the CMC case. Applying his methods to the hee Willmore case, local foliations of Willmore type surfaces have first been constructed simultaneously by Ikoma, Malchiodi, Mondino and Lamm Metzger, Schulze in \cite{malchiodispheres} and \cite{lammfoliation} respectively. The present article addresses whether this is also true in the `domain setting' introduced above.  Perhaps surprisingly, we prove that there is a qualitative difference between the two situations:

\begin{theorem}[Willmore Case]\label{theoremWillmore}
Let $\Omega\subset\R^3$ be a smooth bounded domain and $S:=\partial\Omega$. Denote the Gauß curvature of $S$ by $K^S$ and the scalar mean curvature by $H^S:=\vec H^S\cdot N^S$. Let $p\in S$ be a nondegenerate critical point of $H^S$ and put $v_0:=\frac1{2}\left|\left(\nabla^2 H^S(p)\right)^{-1}\nabla K^ S(p)\right|$.
\begin{enumerate}
    \item If $v_0<1$, there exist $\lambda_0>0$ and a neighborhood $U$ of $p$ in $\bar\Omega$ such that for $\lambda\in(0,\lambda_0)$ the Willmore type surfaces $\phi^{\gamma(\lambda),\lambda}\in\mathcal M_\lambda(S)$ provide a foliation $U\backslash\set p$.
    \item If $v_0>1$, then for no neighborhood $U$ of $p$ in $\bar\Omega$ it is true that $U\backslash\set p$ is foliated by surfaces of Willmore type $\phi_\lambda\in \mathcal M_\lambda(S)$.
\end{enumerate}
\end{theorem}

Similarly, for the CMC case, we establish the following theorem:

\begin{theorem}[CMC Case]\label{theoremCMC}
Let $\Omega\subset\R^3$ be a smooth bounded domain and $S:=\partial\Omega$. Denote the Gauß curvature of $S$ by $K^S$ and the scalar mean curvature by $H^S:=\vec H^S\cdot N^S$. Let $p\in S$ be a nondegenerate critical point of $H^S$ and put $v_0:=\frac1{3}\left|\left(\nabla^2 H^S(p)\right)^{-1}\nabla K^ S(p)\right|$.
\begin{enumerate}
    \item If $v_0<1$, there exist $\lambda_0>0$ and a neighborhood $U$ of $p$ in $\bar\Omega$ such that for $\lambda\in(0,\lambda_0)$ the Willmore type surfaces $\phi^{\gamma(\lambda),\lambda}\in\mathcal N_\lambda(S)$ provide a foliation $U\backslash\set p$.
    \item If $v_0>1$, then for no neighborhood $U$ of $p$ in $\bar\Omega$ it is true that $U\backslash\set p$ is foliated by CMC surfaces $\phi_\lambda\in \mathcal N_\lambda(S)$.
\end{enumerate}
\end{theorem}

By explicit construction, we also verify that in Theorems \ref{theoremWillmore} and \ref{theoremCMC} $v_0$ can take all values in $[0,\infty)$, which implies in particular, that both cases in the two theorems are possible. \\

In the Riemannian setting, the unconditional existence of a foliation follows from the fact that the curve in (2) satisfies $\dot\gamma(0)=0$. Using this terminology, this was first shown by Hofmeister in her Master thesis \cite{hofmesiter}. In Subsection \ref{RiemannianSetting}, we demonstrate how $\dot\gamma(0)=0$ in the Riemannian setting follows from a symmetry argument. Here we also demonstrate that this symmetry argument does not carry over to the `domain setting' that we investigate here. \\

 Surfaces of Willmore type and CMC surfaces have both been widely discussed in the literature.
In addition to the already mentioned articles, we also refer to \cite{hofmesiter}, \cite{malchiodispheres}, \cite{lammfoliation} and \cite{Mondino} for the construction of Willmore type spheres in Riemannian manifolds. A similar approach for the conformal Willmore functional is presented in \cite{ConformalMondino}. In addition to the already mentioned \cite{ye}, we also refer to \cite{PacardCMC} for the construction of CMC spheres in a Riemannian manifold. The study of half-spheres attached to a domain's boundary originates from the aforementioned article by Belletini and Fusco \cite{bellettinifusco} and was adapted by Alessandroni and Kuwert for the Willmore functional in \cite{AK}. The integral methodology of all these works is a form of Laypunov-Schmidt reduction pioneered in the previously mentioned article by Ye \cite{ye}. For a general overview on the construction of small Willmore type spheres via Lyapunov-Schmidt reduction, we refer the reader to \cite{malchiodireview}. A construction of area--constrained critical Willmore spheres via the direct method is achieved in \cite{LammDIrectMethod}. Another type of foliation results where one foliates the asymptotic region of asymptotically flat manifolds by Willmore type or CMC surfaces is established in \cite{lammfoliationsasymp} in the  Willmore and in \cite{YeAsymp} in the CMC case. Finally we remark, that the preprint \cite{montenegro} is in contrast to our result.\\

The proofs for Theorems \ref{theoremWillmore} and \ref{theoremCMC}, respectively, essentially rely on establishing the following Expansions for the Willmore energy and the area functional. Given a critical surface $\phi^{a,\lambda}\in \mathcal M_\lambda(S)$ or $\mathcal N_\lambda(S)$ with barycenter $a\in S$ we prove 
\begin{align}
\mathcal W[\phi^{a,\lambda}]&=2\pi-\lambda\pi H^S(a)+\frac12\pi\left(K^S(a)+\left(\ln(2)-\frac32\right)  H^S(a)^2\right)\lambda^2+\mathcal O(\lambda^3),\label{IntroWillmoreEnergyExp}\\
%---------
\mathcal A[\phi^{a,\lambda}]&=\lambda^2\left(2\pi-\lambda\frac\pi4 H^S(a)+\frac12\pi\left(\frac16K^S(a) -\frac{35}{192} H^S(a)^2\right)\lambda^2+\mathcal O(\lambda^3)\right).\label{IntroAreaEnergyExp}
\end{align}

The main difficulty in establishing these expansions is that the surfaces $\phi^{a,\lambda}$ are only defined implicitly by the implicit function theorem. To establish Equations \eqref{IntroWillmoreEnergyExp} and \eqref{IntroAreaEnergyExp}, we compute explicitly the first non-trivial deviation of $\phi^{a,\lambda}$ from a scaled half-sphere $\lambda\Sp^2_+$ by linearising Equations \eqref{problem} and \eqref{CMCEquationIntro} respectively and subsequently solving the resulting linear PDE's. Thereby, establishing \eqref{IntroWillmoreEnergyExp} and \eqref{IntroAreaEnergyExp} is effectively completely reduced to an explicit computation which we execute by machine calculation.\\

This article is structured as follows: Section \ref{Preliminaries} provides notation and proper definitions. In Section \ref{ConstructionSection}, we outline the construction of Willmore type half-spheres that was developed in \cite{AK}. Additionally, we demonstrate how the reasoning from \cite{AK} can be extended to CMC surfaces and the Riemannian setting. Finally, we demonstrate how the aforementioned symmetry-argument from the Riemannian setting fails to apply here. Section \ref{mainTheoremSection} gives the proof of Theorems \ref{theoremWillmore} and \ref{theoremCMC} while assuming Equations \eqref{IntroWillmoreEnergyExp} and \eqref{IntroAreaEnergyExp}. Also, it is established that both cases in Theorems \ref{theoremWillmore} and \ref{theoremCMC} can occur. In Section \ref{outlinesection}, the derivations of Equations \eqref{IntroWillmoreEnergyExp} and \eqref{IntroAreaEnergyExp} is outlined. A significant part of the relevant computations has been executed with Mathematica \cite{mathematica}. Details on the computations are provided in the \hyperref[foliationsection]{Appendix}. The Mathematica notebooks are provided as supplementary material.
\setcounter{equation}{0}
\section{Preliminaries}\label{Preliminaries}
In the following, let $\alpha\in(0,1)$ be fixed but arbitrary. Additionally
\begin{align*}
    \R^3_+&:=\set{(x_1,x_2,x_3)\in\R^3\ |\ x_3\geq 0}\\
    \Sp^2_+&:=\set{(\omega_1,\omega_2,\omega_3)\in\R^3\ |\ \omega_3\geq 0\textrm{ and }\omega_1^2+\omega_2^2+\omega_3^2=1}.
\end{align*}

\subsection{Terminology}\label{terminilogy}
Let $\tilde g$ be a metric on $\R^3$ close to the euclidean metric $\delta$ in $C^ l$ for large enough $l\in\N$, consider a suitably smooth immersion $f:\Sp^2_+\rightarrow (\R^3,\tilde g)$ and put $g:=f^*\tilde g$. We denote the inner normal ($-\omega$ for the round half-sphere in euclidean space) of $f$ with respect to $\tilde g$ by $\tilde\nu$ and its inner conormal ($e_3$ for the round half-sphere in euclidean space) by $\tilde \eta$. We define the mean curvature $H[f,\tilde g]$ of $f$ with the convention that for the round half-sphere in euclidean space $H=2$. Let $h^ 0$ denote the traceless second fundamental form of $f$ and $\operatorname{Ric}_{\tilde g}$ the Ricci tensor of $(\R^3,\tilde g)$. The scalar Willmore gradient is then given by 
\begin{equation}\label{Willmoreoperatordefinition}
W[f,\tilde g]:=\frac12\left(\Delta_gH+(|h^ 0|^2+\operatorname{Ric}_{\tilde g}(\tilde\nu,\tilde\nu))H\right).
\end{equation}
For a proof, see Theorem 1 in \cite{AK} and note the following difference in conventions: We have included $1/2$ in the definition of $W$. Given a function $F[f,\tilde g]$ that is invariant under reparameterizations (e.g. the area $A$), we denote the $L^2(g)$-gradient along the normal bundle by $\nabla F[f,\tilde g]$. That is, for $\psi:\Sp^2_+\rightarrow\R$
$$\frac d{ds}\bigg|_{s=0} F[f+s\psi\tilde\nu,\tilde g]=\int_{\Sp^2_+}\nabla F[f,\tilde g]\psi d\mu_{g}.$$
As an example, denoting the inclusion $\Sp^2_+\hookrightarrow\R^3$ by $f_0$ we have $\nabla A[f_0,\delta]=-H[f_0,\delta]=-2$.\\

The standard Laplacian on $\Sp^2$ is always denoted by $\Delta$.\\

Let $r:\Sp^2_+\rightarrow\Sp^2_+$ denote the reflection $r(\omega_1,\omega_2,\omega_3)=(-\omega_1,-\omega_2,\omega_3)$. We say that a function $u:\Sp^2_+\rightarrow\R$ is \emph{even} when $u\circ r=u$ and that it is \emph{odd} when $u\circ r=-u$.

\paragraph{Differentials}\ \\
If $f:A_1\times...\times A_k\rightarrow B$ is a differentiable map, then we denote by $D_i f$ the derivative with respect to the $i$-th component. So for example
$$D_1 f(a_1,...,a_k) v:=\frac{d}{dt}\bigg|_{t=0}f(a_1+tv, a_2,...,a_k).$$
The second derivatives are denoted by $D^2$. For example 
$$D_1^2 f(a,b) (v,w):=\frac{d}{dt}\bigg|_{t=0}\frac{d}{ds}\bigg|_{s=0}f(a+tv+sw, b)
\hspace{.2cm}\textrm{or}\hspace{.2cm}
D_{1,2}^2 f(a,b) (v,w):=\frac{d}{dt}\bigg|_{t=0}\frac{d}{ds}\bigg|_{s=0}f(a+tv, b+sw).
$$

\paragraph{Summation Convention}\ \\
We use the following summation convention. Every repeated index is summed over. If the index is Latin, it takes the values $i=1,2$; if it is Greek, it takes the values $\mu=1,2,3$.

\subsection{Blow Up at the Boundary}\label{construction}
In this subsection, we review the setup to construct critical surface achieved in \cite{AK} and collect important formulas for our purposes. For details, we refer to \cite{AK}. 
Let $\Omega\subset\R^3$ be a smooth and bounded domain and denote the interior normal along $S:=\partial\Omega$ by $N^S$. Let $a\in S$ and $b_1, b_2\in T_aS$ be an orthonormal basis of $T_a S$.\\

For $r>0$, let $D_{r}:=\set{x\in\R^2\ |\ |x|<r}$. There exists $r_0=r_0(S)>0$, a neighbourhood $U\subset S$ of $a$ and a smooth function $\varphi^{a}:D_{r_0}\rightarrow \R$, such that 
$$f^a:D_{r_0}\rightarrow S,\ f^p(x):=p+x_1 b_1+x_2 b_2+\varphi^a(x)N^S(a)$$
is a parameterization of $U$. $\varphi^a$ satisfies $\varphi^a(0)=0$, $D\varphi^a=0$ as well as the estimates 
$$|\varphi^a(x)|\leq C(\Omega)|x|^2\hspace{.3cm}\text{and}\hspace{.3cm}|D\varphi^a(x)|\leq C(\Omega)|x|.$$
After potentially shrinking $U$ and $r_0$, we can extend $f^a$ to a diffeomorphism 
$$F^a:D_{r_0}(0)\times(-r_0,r_0)\rightarrow \operatorname{im}(F^a)\subset\R^3,\ F^a(x,z):=f^a(x)+z N^S(a).$$
By compactness of $S$, the radius $r_0$ can be chosen uniformly over all $a\in S$. Let $\lambda_0(S)>0$ so that $2\lambda_0<r_0$. Then, for $\lambda\leq\lambda_0$, the following map is well defined:
$$F^{a,\lambda}:\bar D_2(0)\times [-2,2]\rightarrow\R^3,\ F^{a,\lambda}(x,z):=F^a(\lambda x,\lambda z)$$
Finally, we introduce the scaled-pullback metric 
\begin{equation}\label{metricdefinition}
\tilde g^{a,\lambda}:=\frac1{\lambda^2}\left(F^{a,\lambda}\right)^*\delta.
\end{equation}
Using that $\Omega\in C^\infty$, the discussion leading up to Equation (3.7) in \cite{AK} implies that\\ $\|\tilde g^{a,\lambda}-\delta\|_{C^k(\bar B_2(0)\times[-2,2])}\leq C(k)\lambda$ for all $k\in\N$.

\subsection{Almost Half-Spheres}
For $\Omega$ as described above and $S=\partial\Omega$, we consider the following classes:
\begin{align*} 
\mathcal M^{4,\alpha}(S)&:=\set{\phi\in C^{4,\alpha}(\Sp^2_+,\R^3)\ |\ \text{immersed, $\phi(\partial \Sp^2_+)\subset S$ and $\phi\perp S$ along $\partial \Sp^2_+$}},\\
\mathcal N^{2,\alpha}(S)&:=\set{\phi\in C^{2,\alpha}(\Sp^2_+,\R^3)\ |\ \text{immersed, $\phi(\partial \Sp^2_+)\subset S$ and $\phi\perp S$ along $\partial \Sp^2_+$}}.
\end{align*}
The prototype for surfaces $\phi\in \mathcal M^{4,\alpha}(S),\ \mathcal N^{2,\alpha}(S)$ respectively that are of interest to us are constructed by choosing a $C^{4,\alpha}(\Sp^2_+),\ C^{2,\alpha}(\Sp^2_+)$-small graph function $u:\Sp^2_+\rightarrow\R$ and considering
\begin{equation}\label{findexunot}
F^{a,\lambda}\circ f_u\hspace{.5cm}\textrm{where}\hspace{.5cm}f_u:\Sp^2_+\rightarrow\R^3,\ f_u(\omega):=(1+u(\omega))\omega.
\end{equation}
Such surfaces are referred to as \emph{almost half-spheres}. Closely related to the study of these almost half-spheres is the analysis of the surfaces $f_u$ in $\R^3_+$ equipped with the metric $\tilde g^{a,\lambda}$. In the following, we collect some properties of these surfaces. For these to hold we require that $\|u\|_{C^{4,\alpha}(\Sp^2_+)}<\theta_0$ or $\|u\|_{C^{2,\alpha}(\Sp^2_+)}<\theta_0$ respectively and $\lambda<\lambda_0$, where $\lambda_0$ and $\theta_0$ are bounds that only depend on $S$.\\

The surfaces $F^{a,\lambda}\circ f_u$ are immersed. To belong to the classes $\mathcal M^{4,\alpha}(S)$ or $\mathcal N^{2,\alpha}(S)$ respectively, $f_u$ must meet $\R^2\times 0$ orthogonally with respect $\tilde g^{a,\lambda}$, which is guaranteed, if the interior unit normal $\tilde \nu[u,\tilde g^{a,\lambda}]$ satisfies $\tilde\nu^3[u,\tilde g^{a,\lambda}]=0$ along $\partial \Sp^2_+$ (see Lemma 3 in \cite{AK}). For a general background metric $\tilde g$, the following formula is established in \cite{AK}:
\begin{equation}\label{GeneralNormalFormula}
\tilde\nu[u,\tilde g](\omega)=-\frac{\omega-g^{ij}\tilde g(\omega,\partial_i f_u)\partial_j f_u}{\sqrt{\tilde g(\omega,\omega)-g^{ij}\tilde g(\omega,\partial_i f_u)\tilde g(\omega, \partial_j f_u)}}
\end{equation}
The surfaces $F^{a,\lambda}\circ f_u$ have a well defined surface area, which satisfies 
$$A[F^{a,\lambda}\circ f_u]=\lambda^2 A[f_u,\tilde g^{a,\lambda}].$$
Additionally, the surface $\phi:=F^{a,\lambda}\circ f_u$ separates $\Omega$ into a large exterior part and a small interior part, which we denote by $\Omega_\phi$. We define the \emph{volume} of $\phi$ as 
$$V[\phi]=V[F^{a,\lambda}\circ f_u]=|\Omega_\phi|.$$
Similarly, $f_u$ separates $\R^3_+$ into a large exterior and a small interior domain $\Omega_u$. If $\tilde g$ is a metric on $\R^3_+$, we put 
$$V[f_u,\tilde g]:=\int_{\Omega_u}\sqrt{\det\tilde g}d^3x
\hspace{.5cm}\textrm{so that }\hspace{.5cm}
V[F^{a,\lambda}\circ f_u]=\lambda^3 V[f_u,\tilde g^{a,\lambda}].$$
There exists a nonlinear projection $C$ that maps $F^{a,\lambda}\circ f_u$ to a point $C[F^{a,\lambda}\circ f_u]\in S$, which we refer to as the surfaces (Riemannian) barycenter. This definition coincides with the origin for the round half-sphere attached to $\R^2$. Similarly an analogue projection $C$ for immersions $f_u:\Sp^2_+\rightarrow(\R^3,\tilde g)$ to $\R^2$ can be constructed. These two projections satisfy 
\begin{equation}\label{barycenteridentity}
C[F^{a,\lambda}\circ f_u]=F^{a,\lambda}[ C[f_u,\tilde g^{a,\lambda}]].
\end{equation}
The concept of the Riemannian barycenter is originally due to Karcher \cite{karcher}. We use a slight variant of the local version introduced in \cite{AK}, which is presented in Appendix D in \cite{Metsch}. Finally, we note that in \cite{Metsch}, Appendix D, it is also shown that for small enough $\lambda_0$ and $\theta_0$ each $\phi=F^{a,\lambda}\circ f_u$ may be \emph{parameterized over its barycenter}. That is, there exists a parameterization of the form $\phi=F^{C[\phi],\lambda}\circ f_{\tilde u}$. The parameterization depends on the orthonormal frame chosen at $C[\phi]$ but is unique once a frame is fixed.
For small $\lambda>0$, we define 
\begin{align*} 
\mathcal M^{4,\alpha}_\lambda(S)&:=\left\{\phi\in\mathcal M^{4,\alpha}(S)\ \bigg|\ 
        \begin{array}{l}
            \textrm{$\phi$ is of the form in Equation \eqref{findexunot}, parameterized}\\
            \textrm{over its barycenter and $A[\phi]=2\pi\lambda^2$}
        \end{array}
\right\},\\
\mathcal N^{2,\alpha}_\lambda(S)&:=\left\{\phi\in\mathcal N^{2 ,\alpha}(S)\ \bigg|\ 
        \begin{array}{l}
            \textrm{$\phi$ is of the form in Equation \eqref{findexunot}, parameterized}\\
            \textrm{over its barycenter and $V[\phi]=\frac23\pi\lambda^3$}
        \end{array}
\right\}.
\end{align*}
 Finally, we note that for functionals such as the area, we often write $A[u,\tilde g]:=A[f_u,\tilde g]$ and also $\nabla A[u,\tilde g]:=\nabla A[f_u,\tilde g]=-H[f_u,\tilde g]=:-H[u,\tilde g]$.

\setcounter{equation}{0}
\section{The Construction of Critical Surfaces}\label{ConstructionSection}
In this Section, we outline the arguments by Alessandroni and Kuwert \cite{AK} to construct surfaces of Willmore type. We then show how to adapt their arguments to the CMC case and make some comments on the related Riemannian problem.
\subsection{Construction of Surfaces of Willmore Type}
To solve Equation \eqref{problem}, we make the ansatz $\phi=F^{a,\lambda}\circ f_u$. Using the diffeomorphism $F^{a,\lambda}$ to pull back \eqref{problem} to $(\R^3_+,\tilde g^{a,\lambda})$  yields the Equation 
\begin{equation}\label{pullbackeq1}
\left\{\begin{aligned}
    &W[u,\tilde g^{a,\lambda}]=\alpha H[u,\tilde g^{a,\lambda}],\\
    &B[u,\tilde g^{a,\lambda}]=0,\\
    &A[u,\tilde g^{a,\lambda}]=2\pi.
\end{aligned}    \right.
\end{equation}
$B$ collects the boundary conditions of $\phi$ meeting $S$ orthogonally as well as the third order condition in \eqref{problem}. Following \cite{AK}, it is given by\footnote{Note a slight difference in the definition of the first component. The two definitions are, however, equivalent.}
\begin{equation}\label{BoundaryOperatorDefinition}
B[u,\tilde g^{a,\lambda}]=(B_1[u,\tilde g^{a,\lambda}],B_2[u,\tilde g^{a,\lambda}]):=\left(\omega_3-g^{ij}\tilde g^{a,\lambda}(\omega,\partial_i f_u)\partial_j f_u^3, \frac{\partial H}{\partial\tilde\eta}+H\tilde h^{\R^2}(\tilde\nu,\tilde\nu)\right).
\end{equation}
The philosophy of \cite{AK} is to solve \eqref{pullbackeq1} with the implicit function theorem for an arbitrary abstract background metric $\tilde g$ that is close to the euclidean metric $\delta$ to give $u=u[\tilde g]$. However, due to the translation and scaling invariance of the Willmore energy in euclidean space, the linearized operator has a kernel. To overcome this, the ansatz is modified by prescribing the barycenter as $C[\phi]=a$ or, for an abstract background metric in the pulled-back picture, $C[u,\tilde g]=0$. This leads to the problem
\begin{equation}\label{pullbackeq2}
\left\{\begin{aligned}
    &W[u,\tilde g]=\alpha H[u,\tilde g]+\beta_i\nabla C^i[u,\tilde g],\\
    &B[u,\tilde g]=0,\\
    &A[u,\tilde g]=2\pi,\\
    &C[u,\tilde g]=0.
\end{aligned}\right.    
\end{equation}
The new Lagrange multipliers reflect the addition of the new constraint. 
The two constraints $A=2\pi$ and $C=0$ fix the kernel of the linearized operator. In Lemma 6 in \cite{AK}, it is then demonstrated that the implicit function theorem is now applicable to give a unique solution 
$$(u,\alpha,\beta_i)=(u[\tilde g],\alpha[\tilde g],\beta_i[\tilde g])\textrm{ in a neighbourhood of }(0,0,0,0)\in C^{4,\alpha}(\Sp^2_+)\times\R^3.$$
In particular, the following Theorem is established by taking $\tilde g^{a,\lambda}$ as $\tilde g$. The content of this Theorem is a summary of Lemma 6, Proposition 1 and Theorem 2 as well as its preceding discussion from \cite{AK}:
\begin{theorem}\label{graphfunction}
Let $\Omega\in C^\infty$ be a bounded domain. There exists $\lambda_0(\Omega)$ and a neighbourhood $U$ of $(0,0,0)\in C^{4,\alpha}(\Sp^2_+)\times\R^3$ such that for every $a\in S$ and $\lambda\leq\lambda_0$ there exists a unique solution $(u^{a,\lambda},\alpha^{a,\lambda},\beta_i^{a,\lambda})\in U$ to \eqref{pullbackeq2}. Moreover the map
$$S\times[0,\lambda_0(S)]\ni (a,\lambda)\mapsto (u^{a,\lambda},\alpha^{a,\lambda},\beta_i^{a,\lambda})\in U$$
is smooth. Additionally, $u^{a,\lambda}\in C^\infty(\Sp^2_+)$ and for all $k\geq 4$
\begin{equation}\label{uestimate}
\|u\|_{C^{k,\alpha}}\leq C(\Omega,k)\lambda.
\end{equation}
\end{theorem}
For $a\in S$ and  $\lambda>0$ sufficiently small, the surface
\begin{equation}\label{phiplambdaWillmore}
\phi^ {a,\lambda}:=F^{a,\lambda}\circ f_{u^{a,\lambda}}
\end{equation}
is now a candidate for a solution to \eqref{problem}. However, the solution $u^{a,\lambda}$ from Theorem \ref{graphfunction} is, in general, not a solution to Problem \eqref{pullbackeq1} since in general $\beta_i^{a,\lambda}\neq 0$. To obtain a solution to \eqref{problem} or \eqref{pullbackeq1} respectively one studies the \emph{reduced functional}
$$\mathcal{\bar W}:S\times[0,\lambda_0)\rightarrow\R,\ (a,\lambda)\mapsto \mathcal W[\phi^{a,\lambda}]=\mathcal W[u^{a,\lambda},\tilde g^{a,\lambda}].$$
We paraphrase Theorem 2 from \cite{AK}:
\begin{theorem}\label{AKTheo2}
Let $\Omega\subset\R^3$ be a smooth, bounded domain and $S:=\partial\Omega$. Then for $\lambda\in[0,\lambda_0]$ a point $a\in S$ is critical for $\bar{\mathcal W}(\cdot,\lambda)$ if and only if $\phi^{a,\lambda}$ is a solution to \eqref{problem} or, equivalently, $u^{a,\lambda}$ solves problem \eqref{pullbackeq1}. 
\end{theorem}

Next, we outline how the existence of critical points is established in \cite{AK}. The scalar mean curvature of $S=\partial\Omega$ is defined by $H^S:=\vec H^S\cdot N^S$ with the interior unit normal $N^S$. In Equation (3.10) in \cite{AK}, the expansion 
\begin{equation}\label{AKExpansion}
\bar{\mathcal W}[a,\lambda]=2\pi-\pi H^ S(a)\lambda+\mathcal O(\lambda^2)
\end{equation}
is derived. Using this, the following Theorem is established:
\begin{theorem}\label{theorem3}
Let $\Omega$ be a smooth, bounded domain and $p\in S=\partial\Omega$ be a nondegenerate critical point of $H^ S$. There exists $\lambda_0(S)>0$, a neighbourhood $U\subset S$ of $p$ and a smooth curve $\gamma:[0,\lambda_0)\rightarrow S$ such that $\gamma(0)=p$ and each $\phi^ {\gamma(\lambda),\lambda}$ is a critical point of $\mathcal W$ in $\mathcal M^ {4,\alpha}_\lambda(S)$. Additionally, for $\lambda<\lambda_0$ these are the only critical points of $\mathcal W$ in $\mathcal M^ {4,\alpha}_\lambda(S)$ with barycenter in~$U$.
\end{theorem}

\subsection{Construction of CMC Surfaces}
The original construction of CMC half-spheres with volume constraint originates from the article by Bellettini and Fusco \cite{bellettinifusco}. In this article, we do, however, not follow their construction but copy the strategy of Alessandroni and Kuwert \cite{AK}. Inserting the ansatz $\phi=F^{a,\lambda}\circ f_u$ into \eqref{CMCEquationIntro}  yields the following problem:
\begin{equation}\label{CMCEquationPullback}
\left\{\begin{aligned}
&H[u,\tilde g^{a,\lambda}]=\alpha,\\
&B_1[u,\tilde g^{a,\lambda}]=0,\\
&V[u,\tilde g^{a,\lambda}]=\frac{2\pi}3.
\end{aligned}\right.
\end{equation}
Again, the main idea is to use the implicit function theorem to solve for $u$ in dependence of an abstract background metric $\tilde g$ for which we can choose $\tilde g^{a,\lambda}$ at the end. Again this is not possible since the linearized operator has a kernel that is related to the invariance of the area functional under translations (in the euclidean case). To overcome this problem, we employ the same strategy and study
\begin{equation}\label{CMCEquationPullbackConstrained}
\left\{\begin{aligned}
&H[u,\tilde g]=\alpha+\beta_i\nabla C^i[u,\tilde g],\\
&B_1[u,\tilde g]=0,\\
&V[u,\tilde g]=\frac{2\pi}3,\\
&C^i[u,\tilde g]=0\hspace{.5cm}\textrm{for $i=1,2$}.
\end{aligned}\right.
\end{equation}
In complete analogy to Theorem \ref{graphfunction}, the following Theorem can be established:
\begin{theorem}\label{graphfunctionCMC}
Let $\Omega\in C^\infty$ be a bounded domain. There exists $\lambda_0(\Omega)$ and a neighbourhood $U$ of $(0,0,0)\in C^{2,\alpha}(\Sp^2_+)\times\R^3$ such that for every $p\in S$ and $\lambda\leq\lambda_0$ there exists a unique solution $(u^{a,\lambda},\alpha^{a,\lambda},\beta_i^{a,\lambda})\in U$ to \eqref{CMCEquationPullbackConstrained}. Moreover the map
$$S\times[0,\lambda_0(S)]\ni (a,\lambda)\mapsto (u^{a,\lambda},\alpha^{a,\lambda},\beta_i^{a,\lambda})\in U$$
is smooth. Additionally, $u^{a,\lambda}\in C^\infty(\Sp^2_+)$ and for all $k\geq 2$
\begin{equation}\label{uestimateCMC}
\|u\|_{C^{k,\alpha}}\leq C(\Omega,k)\lambda.
\end{equation}
\end{theorem}
Also, the rest of the analysis closely follows \cite{AK}. First the surface
\begin{equation}\label{phiplambdaCMC}
\phi^ {a,\lambda}:=F^{a,\lambda}\circ f_{u^{a,\lambda}}.
\end{equation}
is introduced as a candidate for a solution to \eqref{CMCEquationIntro}. To realize a solution, one then studies the reduced functional 
$$\bar A(a,\lambda):=A[\phi^{a,\lambda}]=\lambda^2 A[u^{a,\lambda},\tilde g^{a,\lambda}].$$
Following the proof of Theorem 2 in \cite{AK}, the following analog theorem can be established:
\begin{theorem}\label{AKTheo2CMC}
Let $\Omega\subset\R^3$ be a smooth bounded domain and $S:=\partial\Omega$. Then for $\lambda\in[0,\lambda_0]$, a point $a\in S$ is critical for $\bar{A}(\cdot,\lambda)$ if and only if $\phi^{a,\lambda}$ solves \eqref{CMCEquationIntro} or, equivalently,  $u^{a,\lambda}$ solves problem \eqref{CMCEquationPullback}. 
\end{theorem}
In Appendix \ref{CMCOutline}, we establish the following expansion that is the analog of Equation \eqref{AKExpansion} for the CMC case:
\begin{equation}\label{AKExpansionCMC}
A[\phi^ {a,\lambda}]=\lambda^2 A[u^{a,\lambda},\tilde g^{a,\lambda}]=\lambda^2\left(2\pi-\frac\pi4 H^ S(a)\lambda+\mathcal O(\lambda^2)\right)
\end{equation}
Using Equation \eqref{AKExpansionCMC}, we can follow the proof of Theorem 3 in \cite{AK} to establish the following result.
\begin{theorem}\label{theorem3CMC}
Let $\Omega$ be a smooth, bounded domain and $p\in S=\partial\Omega$ a nondegenerate critical point of $H^ S$. There exists $\lambda_0(S)>0$, a neighbourhood $U\subset S$ of $p$ and a smooth curve\\ $\gamma:[0,\lambda_0)\rightarrow S$ such that $\gamma(0)=p$ and each $\phi^ {\gamma(\lambda),\lambda}$ is a critical point of $A$ in $\mathcal N^ {2,\alpha}_\lambda(S)$. Additionally, for $\lambda<\lambda_0$ these are the only critical points of $A$ in $\mathcal N^ {2,\alpha}_\lambda(S)$ with barycenter in $U$.
\end{theorem}
This result has already been established by Bellettini and Fusco in \cite{bellettinifusco}. See Remark 1.4 and Corollary 5.6 in their article. 

\subsection{Key Difference to the Riemannian Setting}\label{RiemannianSetting}
The analysis for constructing critical points to the area or the Willmore functional can be adapted to `\emph{almost spheres}' in a Riemannian manifold $M$. For the Willmore functional, this is, for example, presented in Mattuschkas' Ph.D. -thesis \cite{Mattuschka}. Instead of $F^{a,\lambda}$, we now consider  
\begin{equation}\label{RiemCaseParam}
\mathcal F^{a,\lambda}:\bar B_2(0)\rightarrow M,\ \mathcal F^{a,\lambda}(x):=\exp_a(\lambda x_\mu b_\mu)\hspace{.5cm}\textrm{$b_\mu$: any orthonormal frame of $T_aM$}.
\end{equation}
For a $C^{4,\alpha}$-small graph function $u:\Sp^2\rightarrow\R$, we then consider an \emph{almost sphere} 
\begin{equation}\label{ExpAlmostSphere}
\mathcal F^{a,\lambda}\circ f_u\hspace{.5cm}\textrm{where}\hspace{.5cm}f_u:\Sp^2\rightarrow\R^3,\ f_u(\omega):=(1+u(\omega))\omega.
\end{equation}
In analogy to the class $\mathcal M^{4,\alpha}_\lambda(S)$, we define 
$$
\mathcal M^{4,\alpha}_\lambda(M):=\set{
\phi\textrm{ of the form }\eqref{ExpAlmostSphere}\ |\ A[\phi]=4\pi\lambda^2
}.
$$
The analog of Theorem \ref{graphfunction} -- that is, the existence of critical points with prescribed barycenter -- is established as Theorem 1.0.1 in  \cite{Mattuschka}. Denoting the critical points of area $4\pi\lambda^2$ with prescribed barycenter $a$ as $\Phi^{a,\lambda}$ and the scalar curvature of $M$ by $\operatorname{Sc}$, the Willmore energy has the following expansion:
\begin{equation}\label{mondinoexp}
    \mathcal W[\Phi^{a,\lambda}]=4\pi-\frac{2\pi}3\lambda^2 \operatorname{Sc}(a)+0\cdot\lambda^3+\mathcal O(\lambda^4)
\end{equation}
This is not shown in \cite{Mattuschka}, but by Ikoma, Malchiodi and Mondino in \cite{malchiodispheres} -- though with a different approach than the one we described in Subsection \ref{construction}. Based on this expansion, the following analog of Theorem \ref{theorem3} can be established:
\begin{theorem}\label{RiemCurveTheorem}
    Let $p\in M$ be a nondegenerate critical point of $\operatorname{Sc}$. There exists $\lambda_0(M)>0$, a neighbourhood $U$ of $p$ and a smooth curve $\gamma:[0,\lambda_0)\rightarrow M$ such that $\gamma(0)=p$ and each $\Phi^{\gamma(\lambda),\lambda}$ is a critical point of $\mathcal W$ in $\mathcal M_\lambda^{4,\alpha}(M)$. Additionally, for $\lambda<\lambda_0$, these are the only critical points of $\mathcal W$ in $\mathcal M^{4,\alpha}_\lambda(M)$ with barycenter in $U$. 
\end{theorem}

The fact that the surfaces $\Phi^{a,\lambda}$ yield a foliation of a neighborhood of a nondegenerate critical point $p$ of $\operatorname{Sc}$ is derived from the fact that 
\begin{equation}\label{RiemCaseDerivativeCurve}
\dot\gamma(0)=0.
\end{equation}
This fact is intimately linked to the fact that the $\lambda^3$-coefficient in Equation \eqref{mondinoexp} vanishes. We explain this connection in the proof of Theorem \ref{theoremWillmore} or, more concretely, in the proof of Lemma \ref{velocitylemmaWillmore}. To establish Equation \eqref{RiemCaseDerivativeCurve}, the authors of \cite{malchiodispheres} exploit certain symmetries in the expansion of various geometrical quantities. For details, we point to Lemma 5.1 in \cite{malchiodispheres}.\\

We now give an alternative argument to establish \eqref{RiemCaseDerivativeCurve}. The main point is to investigate how the surface $\Phi^{a,\lambda}$ changes when, instead of the orthonormal frame $b_\mu$, we use another frame $b_\mu'$, related to the old frame via $b_\mu'=T_{\mu\nu}b_\nu$ where $T\in \mathbb O(3)$. Denoting the surface construed with this new frame by $\Phi^{T,a,\lambda}$ and following the arguments of Lemma 10 in \cite{AK} it is relatively easy to prove that 
$$\Phi^{T,a,\lambda}=\Phi^{a,\lambda}\circ T\hspace{.5cm}\textrm{where we also use $T$ as a map $T:\Sp^2\rightarrow\Sp^2$}.$$
That is, $\Phi^{a,\lambda}$ changes only by a reparameterization so that $\Phi^{a,\lambda}$ and $\Phi^{T,a,\lambda}$ describe the same geometrical surface. Next, we recall that for given $a\in M$, the surface $\Phi^{a,\lambda}$ is constructed by the implicit function theorem for all $\lambda\in[0,\lambda_0)$. Restricting to $\lambda\geq 0$ is, however, not necessary and in fact, the surfaces $\Phi^{a,\lambda}$ can be constructed for $\lambda\in(-\lambda_0,\lambda_0)$. Note that changing $\lambda\rightarrow-\lambda$ in Equations \eqref{RiemCaseParam} and \eqref{ExpAlmostSphere} is equivalent to changing the frame by the matrix $-\operatorname{Id}_3$. So
\begin{equation}\label{NegativeLambdaInterpretation}
\Phi^{a,-\lambda}=\Phi^{-\operatorname{Id}_3, a,\lambda}=\Phi^{a,\lambda}\circ(-\operatorname{Id}_3).
\end{equation}
Exploiting the invariance of the Willmore energy under reparameterizations, we obtain 
\begin{equation}\label{WillmoreEnergyEvenRiem}
\bar{\mathcal W}(a,-\lambda)=\mathcal W[\Phi^{a,-\lambda}]=\mathcal W[\Phi^{a,\lambda}\circ(-\operatorname{Id}_3)]=\mathcal W[\Phi^{a,\lambda}]=\bar{\mathcal W}(a,\lambda).
\end{equation}
This identity already ensures the $\lambda^3$-coefficient in Equation \eqref{mondinoexp} to vanish. However, we can deduce Equation \eqref{RiemCaseDerivativeCurve} without resorting to the expansion from Equation \eqref{mondinoexp}. To do so, we first note that also Theorem \ref{RiemCurveTheorem} can be extended to allow for $\lambda\in(-\lambda_0,\lambda_0)$ and then provides a curve $\gamma\in C^\infty((-\lambda_0,\lambda_0), M)$ such that the only critical points of $\mathcal W$ inside $\mathcal M^{4,\alpha}_\lambda(M)$ with barycenter in $U$ are $\Phi^{\gamma(\lambda),\lambda}$. Using Equation \eqref{NegativeLambdaInterpretation} we get that for $\lambda>0$
$$\Phi^{\gamma(-\lambda),-\lambda}=\Phi^{\gamma(-\lambda),\lambda}\circ(-\operatorname{Id}_3).$$
So $\Phi^{\gamma(-\lambda),\lambda}$ is critical. By the uniqueness statement in Theorem \ref{RiemCurveTheorem}, it follows that there exists a reparameterization $\kappa$ of $\Sp^2$ such that $\Phi^{\gamma(-\lambda),\lambda}=\Phi^{\gamma(\lambda),\lambda}\circ\kappa$ and hence by the invariance of the Riemannian barycenter under reparameterization
$$\gamma(-\lambda)=C[\Phi^{\gamma(-\lambda),\lambda}]=C[\Phi^{\gamma(\lambda),\lambda}\circ\kappa]=\gamma(\lambda)$$
from which Equation \eqref{RiemCaseDerivativeCurve} follows immediately.\\

\paragraph{Non Applicability in the Domain Setting}\ \\
The argument outlined above fails to apply for half-spheres on the boundary $S$ of a smooth bounded domain $\Omega$. Recalling that almost half-spheres are of the form 
$$
F^{a,\lambda}[(1+u(\omega))\omega]=a+\lambda(1+u(\omega))(\omega_i b_i+\omega_3 N^S(a))+\varphi^a[\lambda(1+u(\omega))\omega_i e_i]N^S(a)
$$
we see that changing $\lambda\rightarrow-\lambda$ flips not only the chosen frame $b_i$ but also the normal vector $N^S(a)$ -- thereby changing it from the interior to the exterior normal. This results in the surface going from sitting \emph{inside} the domain $\Omega$ to sitting \emph{outside} of it.\\

Hence, allowing for negative $\lambda$ and using symmetry arguments fails in the present setting. An artifact of this is already visible in the expansion \eqref{AKExpansion} that contains terms with even and odd powers in $\lambda$ -- something that would be out ruled by an identity like \eqref{WillmoreEnergyEvenRiem}. In fact, we establish 
$$
\bar{\mathcal W}(a,\lambda)=2\pi-\pi H^S(a)\lambda+\frac12\pi\left(K^S(a)+\left(\ln(2)-\frac32\right)H^S(a)^2\right)\lambda^2+\mathcal O(\lambda^3).
$$
The fact that the $\lambda^2$-coefficient does not vanish essentially implies Theorem \ref{theoremWillmore}.
\setcounter{equation}{0}
\section{Proof of the Main Theorem}\label{mainTheoremSection}
In the following, let $K^S$ denote the Gauß curvature of $S$, $N^S$ the interior pointing normal field along $S$ and $H^S:=\vec H^S\cdot N^S$ denote the scalar mean curvature of $S$. The proof of Theorems \ref{theoremWillmore} and \ref{theoremCMC} is based on the following Lemmas:
\begin{lemma}[Willmore Energy Expansion]\label{eneergyexplemmaWillmore}
Let $a\in S$, $\lambda\in[0,\lambda_0)$ and consider the critical points $\phi^{a,\lambda}$ from Equation \eqref{phiplambdaWillmore}. Then
\begin{equation}\label{energyexpansionWillmore} 
\mathcal W[\phi^{a,\lambda}]=2\pi-\lambda\pi H^S(a)+\frac12\pi\left(K^S(a)+\left(\ln(2)-\frac32\right)  H^S(a)^2\right)\lambda^2+\mathcal O(\lambda^3).
\end{equation}
\end{lemma}

\begin{lemma}[Area Expansion]\label{eneergyexplemmaCMC}
Let $a\in S$, $\lambda\in[0,\lambda_0)$ and consider the critical points $\phi^{a,\lambda}$ from Equation \eqref{phiplambdaCMC}. Then
\begin{equation}\label{energyexpansionCMC} 
\mathcal A[\phi^{a,\lambda}]=\lambda^2\left(2\pi-\lambda\frac\pi4 H^S(a)+\frac12\pi\left(\frac16K^S(a) -\frac{35}{192} H^S(a)^2\right)\lambda^2+\mathcal O(\lambda^3)\right).
\end{equation}
\end{lemma}

Using Equations \eqref{energyexpansionWillmore} and \eqref{energyexpansionCMC} we can derive an explicit formula for $\gamma'(0)$ from Theorems \ref{theorem3} and \ref{theorem3CMC}. 

\begin{lemma}\label{velocitylemmaWillmore}
Let $p\in S$ be a nondegenerate critical point of $H^S$ and $\gamma$ as in Theorem \ref{theorem3}. Then 
$$\dot\gamma(0)=\frac12\left(\nabla^2 H^S(p)\right)^{-1}\nabla K^S(p).$$
\end{lemma}
\begin{proof}
Inspecting the proof of Theorem \ref{theorem3} (Theorem 3 in \cite{AK}) shows the existence of $\gamma$ is derived by studying  
\begin{equation}\label{definitionofv}
    v(a,\lambda):=\left\{\begin{aligned}
 \frac1\lambda\nabla_1 \bar{\mathcal W}[a,\lambda]\hspace{.8cm} &,\ \textrm{for $\lambda\neq 0$},\\
\frac{\partial}{\partial\lambda}\bigg|_{\lambda=0}\nabla_1 \mathcal W[a,\lambda]&,\ \textrm{for $\lambda=0$}.
\end{aligned}\right.
\end{equation}
We recall $\bar{\mathcal W}[a,\lambda]:=\mathcal W[\phi^ {a,\lambda}]$ is a smooth function satisfying $\bar{\mathcal W}[a,0]\equiv 2\pi$. Hence $\nabla_1 \bar {\mathcal W}[a,0]\equiv 0$ and we can write 
$$v(a,\lambda)=\frac1\lambda(\nabla_1\bar{\mathcal W}[a,\lambda]-\nabla_1\bar{\mathcal W}[a,0])=\int_0^1(\partial_2 \nabla_1\bar{\mathcal W})(a,\lambda s)ds.$$
So $v$ is smooth and, by using Equation \eqref{energyexpansionWillmore}, we get $v(a,0)=-\pi\nabla H^  S(a)$ and thus $\nabla_1 v(p,0)=-\pi\nabla_1^2 H^  S(p)$. Clearly $v(p,0)=0$ and, since $p$ in nondegenerate, $\nabla_1 v(p,0)$ is invertible. The implicit function theorem gives a $\lambda_0>0$, a neighbourhood $U$ of $p$ and a $C^{\infty}$ curve $\gamma:[0,\lambda_0)\rightarrow S$ satisfying $\gamma(0)=p$ and
$$v(a,\lambda)=0\hspace{.2cm}\textrm{for $\lambda<\lambda_0$ and $a\in U$}\hspace{.5cm}\Leftrightarrow\hspace{.5cm}a=\gamma(\lambda).$$ 
In particular, this establishes Theorem \ref{theorem3}. By construction $v(\gamma(\lambda),\lambda)=0$. Differentiating with respect to $\lambda$ at $\lambda=0$ and using Equation \eqref{energyexpansionWillmore} shows 
$$\pi\nabla^2 H^ S(\gamma(0))\dot\gamma(0)=-\nabla_1 v(\gamma(0),0)\dot\gamma(0)=\frac{\partial v(p,\lambda)}{\partial\lambda}\bigg|_{\lambda=0}=\frac{\pi}{2}\nabla K^ S(p).$$
\end{proof}

Following the same argument, we establish the following analog Lemma for the CMC case.
\begin{lemma}\label{velocitylemmaCMC}
Let $p\in S$ be a nondegenerate point of $H^S$ and $\gamma$ as in Theorem \ref{theorem3CMC}. Then 
$$\dot\gamma(0)=\frac13\left(\nabla ^2 H^S(p)\right)^{-1}\nabla K^S(p).$$
\end{lemma}

Lemmas \ref{velocitylemmaWillmore} and \ref{velocitylemmaCMC} essentially imply Theorems \ref{theoremWillmore} and \ref{theoremCMC}. We quickly give the intuitive argument. Consider the surfaces $\phi^{\gamma(\lambda),\lambda}$. Using the diffeomorphism $F^{p}$, we can pull these back to $\R^3_+$, which produces surfaces that are essentially of the form
$$\Phi^0_\lambda:\Sp^2_+\rightarrow\R^3_+,\ \Phi^0_\lambda(\omega):=\lambda v_0 e_1+\lambda\omega$$
where $v_0$ is given by Lemmas \ref{velocitylemmaWillmore} and \ref{velocitylemmaCMC} respectively. It is easy to see that for $v_0<1$, these are pairwise disjoint and that for $v_0>1$, this is not the case. This is perhaps best seen in a picture:
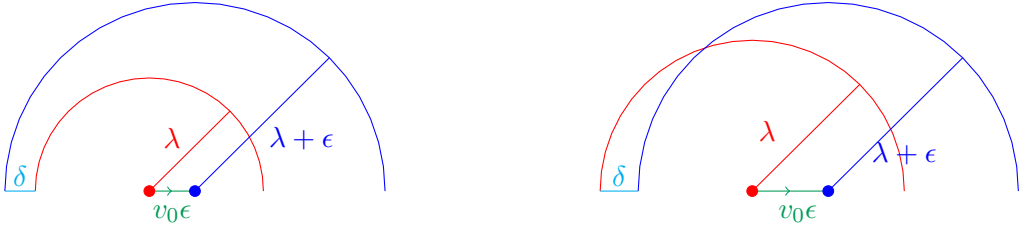
\begin{figure}[htbp]
\begin{minipage}{.5\textwidth}
\begin{center}
\begin{tikzpicture}
\draw [red, xshift=2.5cm, domain=0:180] plot(\x:1.5cm);
\draw [blue, xshift=3.1cm, domain=0:180] plot(\x:2.5cm);
\draw[blue] (3.1cm,0)--++(45:2.5cm);
\draw[red] (2.5,0)--++(45:1.5);
\draw[ForestGreen,->] (2.5,0)--(2.8,0);\draw[ForestGreen](2.8,0)--(3.1,0);
\draw[cyan] (.6,0)--(1,0);
\node at (4.5cm,0.7) {$\color{blue}\lambda+\epsilon$};
\node at (2.8,.7) {$\color{red}\lambda$};
\node at (2.8,-.3) {$\color{ForestGreen}v_0\epsilon$};
\node at (.8,.2) {$\color{cyan}\delta$};
\draw[red, xshift=2.5cm, fill] (0,0) circle (2pt);
\draw[blue, xshift=3.1cm, fill] (0,0) circle (2pt);
\end{tikzpicture}
\end{center}
\end{minipage}
%----------
\begin{minipage}{.5\textwidth}
\begin{center}
\begin{tikzpicture}
\draw [red, xshift=2.5cm, domain=0:180] plot(\x:2cm);
\draw [blue, xshift=3.5cm, domain=0:180] plot(\x:2.5cm);
\draw[blue] (3.5cm,0)--++(45:2.5cm);
\draw[red] (2.5,0)--++(45:2);
\draw[ForestGreen,->] (2.5,0)--(3,0);\draw[ForestGreen](3,0)--(3.5,0);
\draw[cyan] (.5,0)--(1,0);
\node at (4.5cm,0.5) {$\color{blue}\lambda+\epsilon$};
\node at (2.7,.8) {$\color{red}\lambda$};
\node at (3.1,-.3) {$\color{ForestGreen}v_0\epsilon$};
\node at (.75,.2) {$\color{cyan}\delta$};
\draw[red, xshift=2.5cm, fill] (0,0) circle (2pt);
\draw[blue, xshift=3.5cm, fill] (0,0) circle (2pt);
\end{tikzpicture}
\end{center}
\end{minipage}
\caption{Foliation property in unperturbed case}
\label{figure}
\end{figure}\ \\
The leftmost point of the blue circle is at $v_0(\lambda+\epsilon)-(\lambda+\epsilon)$ and the leftmost point of the red circle is as $\lambda v_0-\lambda$. hence, their signed distance is $\delta=\lambda-(\lambda+\epsilon)+v_0\epsilon=(v_0-1)\epsilon$. So, if $v_0>1$ (right picture), the circles move so fast to the right that circles of different radii intersect, while this does not happen for $v_0<1$ (left picture). In the following Theorem, we make this intuitive argument rigorous.

\begin{theorem}\label{maintheoremWillmore}
Let $\Omega$ be a smooth bounded domain, $p\in S=\partial\Omega$ be a nondegenerate critical point of $H^S$, $\phi^{\gamma(\lambda),\lambda}$ denote the surfaces of Willmore type from Theorem \ref{theorem3} and put $ v_0:=|\dot\gamma(0)|$.
\begin{enumerate}
    \item If $v_0<1$, there exists $\lambda_0(v)>0$ and a neighborhood $U$ of $p$ in $\bar\Omega$ such that the surfaces  $(\phi^{\gamma(\lambda),\lambda})_{0<\lambda<\lambda_0}$ provide a foliation of $U\backslash\set p$. 
    \item If $v_0>1$, then for no neighborhood $U$ of $p$ in $\bar\Omega$ and $\lambda_0>0$ it is true that the surfaces  $(\phi^{\gamma(\lambda),\lambda})_{0<\lambda<\lambda_0}$ provide a foliation of $U\backslash\set p$. 
\end{enumerate}
\end{theorem}
As the area-constrained Willmore surfaces close to $p$ are precisely the surfaces $\phi^ {\gamma(\lambda),\lambda}$ it follows that in the second case, no other foliation of area-constrained Willmore surfaces is possible.\\

In Subsection \ref{section3}, we prove that both cases are possible.

\begin{proof}
Throughout this proof, $R(\lambda,\omega)$ denotes a smooth function in $\lambda$ and $\omega$ that may change from line to line.
    Let $U$ be a small neighbourhood of $p$ and $b_1,b_2:U\rightarrow TS$ such that $(b_1(a), b_2(a))$ is an orthonormal basis of $T_a S$ for all $a\in U$. Additionally, we may assume that $\dot\gamma(0)=v_0 b_1(p)$ with $v_0\in[0,\infty)$. We write $b_3(a):=N^S(a)$ so that for small enough $\lambda>0$
\begin{equation}\label{Foliationeq01}
\phi^{\gamma(\lambda),\lambda}=\gamma(\lambda)+\lambda(1+u^{\gamma(\lambda),\lambda}(\omega))\omega_\mu b_\mu (\gamma(\lambda))+\varphi^{\gamma(\lambda)}(\lambda(1+u^{\gamma(\lambda),\lambda}(\omega))\omega_i e_i)b_3(\gamma(\lambda)).
\end{equation}
First, we use Taylor's formula to write
\begin{equation}\label{Foliationeq02}
\gamma(\lambda)=p+\lambda\dot\gamma(0)+\lambda^2\int_0^1 \gamma''(\lambda s) (1-s) ds=p+\lambda\dot\gamma(0)+\lambda^2 R(\lambda,\omega).
\end{equation}
Recalling Theorem \ref{graphfunction}, we define $u:[0,\lambda_0]\rightarrow C^{4,\alpha}(\Sp^2_+)$ by $u(t):=u^{\gamma(\lambda), t}$. By Theorem \ref{graphfunction}, $u$ is smooth and satisfies $u(0)=0$. Hence 
\begin{equation}\label{Foliationeq03}
\lambda u^{\gamma(\lambda),\lambda}=\lambda(u(\lambda)-u(0))=\lambda^2 \int_0^1 u'(\lambda s)ds=\lambda^2 R(\lambda,\omega).
\end{equation}
Next, using $\varphi^a(0)=0$ and $\nabla \varphi^a(0)=0$, we use similar arguments to deduce 
\begin{equation}\label{Foliationeq04}
\varphi^p(\lambda \omega_i e_i)b_3(p)=\lambda^2 R(\lambda,\omega)
\hspace{.5cm}\textrm{and}\hspace{.5cm}
\varphi^{\gamma(\lambda)}(\lambda(1+u^{\gamma(\lambda),\lambda}(\omega))\omega_i e_i)b_3(\gamma(\lambda))=\lambda^2 R(\lambda,\omega).
\end{equation}
Finally, we use $\gamma(0)=p$ to write
\begin{equation}\label{Foliationeq05}
\lambda b_\mu(\gamma(\lambda))=\lambda b_\mu(p)+\lambda^2\int_0^1 \nabla_{\dot\gamma(s\lambda)}b_\mu(\gamma(s\lambda))ds=\lambda b_\mu(p)+\lambda^2 R(\lambda,\omega).
\end{equation}
Inserting Equations \eqref{Foliationeq02}, \eqref{Foliationeq03}, \eqref{Foliationeq04} and \eqref{Foliationeq05} into Equation \eqref{Foliationeq01}, we get
$$
\phi^{\gamma(\lambda),\lambda}(\omega)=p+\lambda\dot\gamma(0)+\lambda \omega_\mu b_\mu(p)+\lambda^2 R[\lambda,\omega]=F^p[\lambda v_0 e_1+\lambda \omega_\mu e_\mu]+\lambda^2 R(\lambda,\omega).
$$
As $F^p$ is a smooth diffeomorphism, we deduce 
$$
\left(F^p\right)^{-1}(\phi^{\gamma(\lambda),\lambda })=\lambda v_0 e_1+\lambda\omega+\lambda^2 R(\lambda,\omega).
$$
Theorem \ref{foliationsection} implies that the surfaces $\left(F^p\right)^{-1}(\Phi^{\gamma(\lambda),\lambda})$ provide a foliation of a deleted neighborhood of $0$ if $v_0<1$ and that this is not the case, when $v_0>1$. Since $F^p$ is a diffeomorphism, the Theorem follows.
\end{proof}

Following essentially the same argument, we establish:
\begin{theorem}\label{maintheoremCMC}
Let $\Omega$ be a smooth bounded domain, $p\in S=\partial\Omega$ be a nondegenerate critical point of $H^S$, $\phi^{\gamma(\lambda),\lambda}$ denote the CMC surfaces from Theorem \ref{theorem3CMC} and put $ v_0:=|\dot\gamma(0)|$.
\begin{enumerate}
    \item If $v_0<1$, there exists $\lambda_0(v)>0$ and a neighborhood $U$ of $p$ in $\bar\Omega$ such that the surfaces  $(\phi^{\gamma(\lambda),\lambda})_{0<\lambda<\lambda_0}$ provide a foliation of $U\backslash\set p$. 
    \item If $v_0>1$, then for no neighborhood $U$ of $p$ in $\bar\Omega$ and $\lambda_0>0$ it is true that the surfaces  $(\phi^{\gamma(\lambda),\lambda})_{0<\lambda<\lambda_0}$ provide a foliation of $U\backslash\set p$. 
\end{enumerate}
\end{theorem}
Again we remark that the volume-constrained CMC surfaces close to $p$ are precisely the surfaces $\phi^ {\gamma(\lambda),\lambda}$ it follows that in the second case, no other foliation of volume-constrained CMC surfaces is possible.\\

\subsection{Construction of Both Cases}\label{section3}
We show by explicit construction that both cases in Theorems \ref{maintheoremWillmore} and \ref{maintheoremCMC} are possible. Let $D_r(0)$ denote the open ball in $\R^2$ with radius $r$, which is centered at $0$. For $u\in C^\infty(\bar D_1(0)$ we consider the surface defined by the graph 
\begin{equation}\label{domainconstructioneq1}
\Phi_u:\bar D_1(0)\rightarrow\R^3,\ \Phi_u(x,y):=(x,y,u(x,y)).
\end{equation}
For any such surface, there exists a smooth and bounded domain $\Omega\subset\R^3$ such that $\Phi_u(\bar D_{\frac12}(0))\subset\partial\Omega=:S$. We put $p:=\Phi(0,0)$ and assume that $p$ is a nondegenerate critical point, the mean curvature $H^S$ of $S$. That is 
$$\nabla H^S(p)=0\hspace{.5cm}\textrm{and}\hspace{.5cm}\nabla^2 H^S(p)\textrm{ is invertible}.$$
Denoting the Gauß curvature of $S$ by $K^S$, we define 
$$v_0:=\left(\nabla^2 H^S(p)\right)^{-1}\nabla K^S(p)\in T_p S.$$
We work in the local coordinate system defined by $\Phi_u$ in Equation \eqref{domainconstructioneq1}. Then
$$\left(\nabla^2 H^S\right)_{ij} g^{ik}v_0^j=g^{kj}\partial_j K^S.$$
In general, the covariant Hessian has coordinates $(\nabla^2 H^S)_{ij}=\partial_{ij}H^S-\Gamma^k_{\ ij}\partial_k H^S$. However, since $(0,0)$ is a critical point of $H^S$, we have 
\begin{equation}\label{domainconstructioneq2}
\nabla^2 H^S(0,0)=\begin{bmatrix}
\partial_{xx} H^S(0,0) & \partial_{xy} H^S(0,0)\\
\partial_{xy} H^S(0,0) & \partial_{yy} H^S(0,0)
\end{bmatrix}.\end{equation}
Therefore, the components $v_{0,x}$ and $v_{0,y}$ of $v_0$ in the coordinates $\Phi_u$ are given by 
\begin{equation}\label{domainconstructioneq3}
\vec v_0:=\begin{bmatrix}
    v_{0,x}\\
    v_{0,y}
\end{bmatrix}
=
\begin{bmatrix}
\partial_{xx} H^S(0,0) & \partial_{xy} H^S(0,0)\\
\partial_{xy} H^S(0,0) & \partial_{yy} H^S(0,0)
\end{bmatrix}^{-1}
\begin{bmatrix}
    \partial_x K(0,0)\\
    \partial_y K(0,0)
\end{bmatrix}.
\end{equation}
Finally, we relate the euclidean norm of the component vector $\vec v_0$ with the norm of $v_0$. Writing 
$$v_0=v_{0,x}\frac{\partial \Phi_u}{\partial x}\bigg|_{(0,0)}+v_{0,y}\frac{\partial \Phi_u}{\partial y}\bigg|_{(0,0)},$$
it is then readily seen that
\begin{align*} 
|v_0|^2&=
v_{0,x}^2 \left(1+(\partial_ x u(0,0))^2\right)
+v_{0,y}^2\left(1+(\partial_ y u(0,0))^2\right) +2v_{0,x}v_{0,y}\partial_x u(0,0)\partial_y u(0,0)\\
&=v_{0,x}^2+v_{0,y}^2+(v_{0,x}\partial_ x u(0,0)+v_{0,y}\partial_ y u(0,0))^2.
\end{align*}
This shows the following two inequalities:
\begin{equation}\label{v0NormRelation}
|\vec v_0|\leq |v_0|\leq |\vec v_0|\sqrt{1+|\nabla u(0,0)|^2}
\end{equation}

All following computations are executed with Mathematica (see Appendix \ref{mathematicacodeAppendix}). Let $a,c_1$ and $c_2$ denote real parameters. We consider the function
\begin{equation}\label{ConcreteuChoice}
u(x,y)=ax+ a y+xy-c_1x^3-c_2 y^3.
\end{equation}
It is straightforward to check that 
\begin{align*}
\frac{\partial H^S}{\partial x}\bigg|_{x=y=0}&=-\frac{2(a-a^3+3c_1+9a^2c_1+6a^4 c_1)}{(1+2a^2)^{\frac52}},
\\
\frac{\partial H^S}{\partial y}\bigg|_{x=y=0}&=-\frac{2(a-a^3+3c_2+9a^2c_2+6a^4 c_2)}{(1+2a^2)^{\frac52}}.
\end{align*}
Demanding that $(0,0)$ is a critical point of $H^S$ produces linear equations for $c_1$ and $c_2$ with the solution
\begin{equation}\label{Choiceforc1c2}
c_1=c_2=\frac{-a+a^3}{3+9a^2+6a^4}.
\end{equation}
From now on we choose $c_1$ and $c_2$ as in Equation \eqref{Choiceforc1c2}. Using Equation \eqref{domainconstructioneq2} and computing the gradient of the Gauß curvature gives:
\begin{align*} 
\nabla^2 H^S(0,0)&=-\frac{2}{(1+2a^2)^{\frac72}}
\begin{bmatrix}
\frac{a^2(-5-20a^2+a^4)}{1+a^2} & 1+4a^2+a^4\\
1+4a^2+a^4 & \frac{a^2(-5-20a^2+a^4)}{1+a^2}
\end{bmatrix}
\\
\nabla K^S(0,0)&=\frac{4a}{(1+2a^2)^3}\begin{bmatrix}
1\\
1
\end{bmatrix}
\end{align*}
Restricting to $a>0$, it is straightforward to check that
\begin{align}
|\vec v_0|&=\left|(\nabla^2 H^S(0,0))^{-1}\nabla K^S(0,0)\right|
=\left|-\frac{2a(1+a^2)\sqrt{1+2a^2}}{1-15 a^4+2a^6}
\begin{bmatrix}
1 \\ 1\end{bmatrix}\right|\nonumber\\
&=\frac{2a(1+a^2)\sqrt {2(1+2a^2)}}{|1-15 a^4+2a^6|}.\label{speedEndResult}
\end{align}

For $a=0$, we get $\vec v_0=0$ and in view of Equation \eqref{v0NormRelation} also $v_0=0$. This establishes the first case in Theorems \ref{maintheoremWillmore} and \ref{maintheoremCMC}. Next, we note that that the polynomial $p(x)=1-15 x^4+2x^6$ has a unique zero $\xi\in(.5,.52)$. Hence $|\vec v_0|\rightarrow\infty$ for $a\rightarrow\xi^+$, which establishes the second case in both Theorems.

\setcounter{equation}{0}
\section{Outline for the Expansions in \ref{eneergyexplemmaWillmore} and \ref{energyexpansionCMC}}\label{outlinesection}
Let $h$ denote the second fundamental form of $S=\partial\Omega$, $a\in S$ and $(b_1,b_2)$ be an orthonormal basis of $T_a S$ consisting of the eigenvectors of the Weingarten operator. We consider the parameterization 
$$f^a:U\subset\R^2\rightarrow S,\ f[a,x]:=a+x_ib_i+\varphi^a(x)N^S(a)$$
from Subsection \ref{construction}. $f^a(0)=a$ and the first and second fundamental form have the following matrix representations:
$$g_{ij}(0)=\begin{bmatrix}
1 &  \\
 & 1
\end{bmatrix}
\hspace{.5cm}\textrm{and}\hspace{.5cm}
h_{ij}(0)=\begin{bmatrix}
\kappa_1 &  \\
 & \kappa_2
\end{bmatrix}
$$
Here $\kappa_1$ and $\kappa_2$ denote the eigenvalues of the Weingarten operator. We have the following formulas for the mean curvature and Gauß curvature of $S$:
$$H^ S(a)=\kappa_1+\kappa_2\hspace{.5cm}\textrm{and}\hspace{.5cm}K^S(a)=\kappa_1 \kappa_2$$
We abbreviate $\tilde g(\lambda):=\tilde g^{a,\lambda}$. In \cite{AK} (Section 3), it is shown that 
\begin{equation}\label{metricderivativeformula}
\tilde g'(0)=\begin{bmatrix}
 & & \kappa_1x_1\\
  & & \kappa_2x_2\\
\kappa_1x_1 & \kappa_2x_2 & 0
\end{bmatrix}.
\end{equation}

\subsection{Proof of Theorem \ref{eneergyexplemmaWillmore}}\label{Outline}
 Considering the expansion in Equation (\ref{AKExpansion}) that has already been established in \cite{AK}, it suffices to compute 
$$\frac{d^2}{d\lambda^2}\bigg|_{\lambda=0}\mathcal W[\phi^{a,\lambda}]=\frac{d^2}{d\lambda^2}\bigg|_{\lambda=0}\mathcal W[u(a,\lambda),\tilde g^ {a,\lambda}]$$
to prove Lemma \ref{eneergyexplemmaWillmore}. 
We abbreviate $u(\lambda):=u(a,\lambda)$ and write
\begin{equation}\label{WillSecDer}
    \frac{d^2}{d\lambda^2}\bigg|_{\lambda=0}\mathcal W[u(\lambda),\tilde g(\lambda)]=D_1 \mathcal W[0,\delta] u''(0)+D_2 \mathcal W[0,\delta]\tilde g''(0)+\frac{d^2}{d\lambda^2}\bigg|_{\lambda=0}\mathcal W[\lambda u'(0),\delta+\lambda \tilde g'(0)].
\end{equation}
The last term can be computed by deriving explicit formulas for $\tilde g'(0)$ and $u'(0)$. $\tilde g'(0)$ is given in Equation \eqref{metricderivativeformula}. The function $u'(0)$ can be computed analytically (see Appendix \ref{explicitsolutionapp}) and is given by 
\begin{equation}\label{solutionu}
\hspace{-.5cm}u'(0)(\omega)=\frac{\kappa_1-\kappa_2}4\frac{\omega_1^2-\omega_2^2}{1+\omega_3}+(\kappa_1+\kappa_2)\left[1-\ln(2)+\frac{\ln(1+\omega_3)}2-\frac34\omega_3\right]-\frac{\kappa_1\omega_1^2+\kappa_2\omega_2^2}2\omega_3.
\end{equation}
Using these two formulas, $\mathcal W[\lambda u'(0),\delta+\lambda \tilde g'(0)]$ is, in principal, a completely explicit object. A direct machine computation shows
\begin{align}
    &D_1^2 \mathcal W[0,\delta][u'(0), u'(0)]=-\frac{8\pi}7 K^S(a)+\left(\frac{863}{280}-\ln(8)\right)\pi H^S(a)^2,\label{1}\\
    &D_{12}^2\mathcal W[0,\delta][u'(0),\tilde g'(0)]=\frac{23\pi}{14}K^S(a)-\frac{291}{560}\pi H^S(a)^2,\label{2}\\
    &D_2^2\mathcal W[0,\delta][\tilde g'(0),\tilde g'(0)]=\frac{4\pi}{21}K^S(a)+\frac{16\pi}{35} H^S(a)^2.\label{3}
\end{align}
The first two terms in Equation \eqref{WillSecDer} are computed in Appendix \ref{secondderivativeWillmore} and are given by 
\begin{align}
    &D_1\mathcal W[0,\delta] u''(0)=4\pi(\ln(2)-1) H^S(a)^2,\label{4}\\
    &D_2\mathcal W[0,\delta]\tilde g''(0)=-\frac{4\pi}3K^S(a).\label{5}
\end{align}
Combining Equations \eqref{WillSecDer} and \eqref{1}-\eqref{5} gives 
$$\frac{\partial^2}{\partial\lambda^2}\bigg|_{\lambda=0}\mathcal W[u(\lambda), \tilde g(\lambda)]=\pi K^S(a)+\pi\left(\ln(2)-\frac32\right)H^S(a)^2$$
and hence Lemma \ref{energyexpansionWillmore}.

\subsection{Proof of Theorem \ref{energyexpansionCMC}}\label{CMCOutline}
Let $a\in S$. We abbreviate  $u(\lambda):=u(a,\lambda)$ and establish the expansion 
$$\mathcal A[\phi^{a,\lambda}]=\lambda^2 A[u(\lambda),\tilde g(\lambda)]=\lambda^2\left(2\pi-\lambda\frac\pi4 H^S(a)+\frac12\pi\left(\frac16K^S(a) -\frac{35}{192} H^S(a)^2\right)\lambda^2+\mathcal O(\lambda^3)\right)$$
by proving the following two formulas:
\begin{align}
&\frac d{d\lambda}\bigg|_{\lambda=0} A[u(\lambda),\tilde g(\lambda)]=-\frac\pi4 H^S(a)\label{areafirstderivative}\\
&\frac {d^2}{d\lambda^2}\bigg|_{\lambda=0} A[u(\lambda),\tilde g(\lambda)]=\pi\left(\frac16K^S(a) -\frac{35}{192} H^S(a)^2\right)\label{areasecondderivative}
\end{align}
To do this, we write
\begin{align}
&\frac d{d\lambda}\bigg|_{\lambda=0}\hspace{-.1cm}  A[u(\lambda),\tilde g(\lambda)]=D_1 A[0,\delta]u'(0)+D_2 A[0,\delta]\tilde g'(0),\label{areafirstderivative2}\\
&\frac {d^2}{d\lambda^2}\bigg|_{\lambda=0}\hspace{-.1cm} A[u(\lambda),\tilde g(\lambda)]= D_1 A[0,\delta] u''(0)+D_2 A[0,\delta] \tilde g''(0)+\frac{d^2}{d\lambda^2}\bigg|_{\lambda=0} A[\lambda u'(0), \delta+\lambda \tilde g'(0)].\label{areasecondderivative2}
\end{align}
In Equations \eqref{D2AEndResult} and \eqref{CMCAreaLinearizedPHI} in the Appendix we prove that 
\begin{equation}\label{DAResultSummary}
D_1 A[0,\delta] u'(0)=0
\hspace{.5cm}\textrm{and}\hspace{.5cm}
D_2 A[0,\delta]\tilde g'(0)=-\frac\pi4 H^S(a).
\end{equation}
These two identities already establish Equation \eqref{areafirstderivative}. Next, we discuss Equation \eqref{areasecondderivative}.
The last term in Equation Equation \eqref{areasecondderivative} can be computed by deriving explicit formulas for $\tilde g'(0)$ and $u'(0)$. $\tilde g'(0)$ is given in Equation \eqref{metricderivativeformula}. In Appendix \ref{CMCExplicitSolution} we prove that $u'(0)$ is given by 
\begin{equation}\label{CMCuPrimeFormula}
u'(0)(\omega)=\frac{\kappa_1+\kappa_2}4\left(\frac34-\omega_3\right)+\frac{\kappa_1-\kappa_2}4\frac{\omega_1^2-\omega_2^2}{3}\frac{2-3\omega_3+\omega_3^3}{(1-\omega_3^2)^2}
-\frac12\omega_3(\kappa_1\omega_1^2+\kappa_2\omega_2^2).
\end{equation}
Using Equations \eqref{metricderivativeformula} for $\tilde g'(0)$ and \eqref{CMCuPrimeFormula} for $u'(0)$ the following computations can be outsourced to a machine:
\begin{align}
&D_{12}^2 A[0,\delta](u'(0), \tilde g'(0))=\frac{5\pi}{14}K^S(a)-\frac{579\pi}{2240}H^S(a)^2\label{CMC1}\\
&D_{1}^2 A[0,\delta](u'(0), u'(0))=-\left(\frac{31}{270}+\frac49\ln(2)\right)K^S(a)\pi+\left(\frac{2201}{8640}+\frac{\ln(2)}9\right)H^S(a)^2\label{CMC2}\\
&D_2^2A[0,\delta](\tilde g'(0),\tilde g'(0))=\frac{64\pi}{105} K^S(a)-\frac{4\pi}{21} H^S(a)^2\label{CMC3}
\end{align}
In Appendix \ref{secondDerAppCMC} we compute
\begin{align}
&D_2 A[0,\delta]\tilde g''(0)=-\frac{4\pi}{5}K^S(a)+\frac{4\pi}{15}H^S(a)^2,\label{CMC4}\\
&D_1 A[0,\delta]u''(0)=\left(\frac{113}{30240}-\frac{\ln(2)}{9}\right)\pi H^S(a)^2+\left(-\frac{229}{945}+\frac49\ln(2)\right)\pi K^S(a).\label{CMC5}
\end{align}
Combining Equations \eqref{CMC1}, \eqref{CMC2}, \eqref{CMC3}, \eqref{CMC4} and \eqref{CMC5} gives Lemma \ref{eneergyexplemmaCMC}.

\renewcommand{\theequation}{\thesection.\arabic{equation}}
\appendix
\setcounter{equation}{0}
\section{Abstract Foliation Argument}\label{foliationsection}
\begin{theorem}\label{foliationendresult}
Let $v\in[0,\infty)$, $\Lambda>0$ and $f\in C^ \infty([0,\Lambda)\times\Sp^2_+,\R^3_+)$ such that $f[\lambda,\partial\Sp^2_+]\subset\R^2\times 0$. We put 
\begin{equation}\label{abstractsurfaces}
\phi_\lambda(\omega):=\lambda ve_1+\lambda\omega+\lambda^2f[\lambda,\omega].
\end{equation}
If $v\in[0,1)$, then there exists a neighbourhood $U\subset\R^3_+$ of 0 and $\lambda_0\in(0,\Lambda)$ such that the surfaces $(\phi_\lambda)_{0<\lambda<\lambda_0}$ provide a smooth foliation of $U\backslash\set0$. If $v>1$, then the same is not true for any choice of $U$ and $\lambda_0$. In fact, for any $\lambda_0>0$ there exist $0<\lambda_1<\lambda_2<\lambda_0$ such that $\operatorname{im}(\phi_{\lambda_1})\cap \operatorname{im}(\phi_{\lambda_2})\neq\emptyset$.
\end{theorem}

The proof of Theorem \ref{foliationendresult} is split into several lemmas.

\begin{lemma}\label{foliationLemma00}
There exists $\lambda_0>0$ such that for $\lambda\in(0,\lambda_0)$ the maps $\phi_\lambda$ are injective. For such $\lambda,$ the surface $\phi_\lambda$ split $\R^3_+$ into a bounded \emph{interior part} $\mathcal I_{\lambda}$, an unbounded exterior part $\mathcal E_{\lambda}$ and the \emph{boundary}  $S_\lambda:=\operatorname{im}(\phi_\lambda)$. Additionally, if $c:[0,1]\rightarrow\R^3_+$ is continuous with $c(0)\in\mathcal I_\lambda$ and $c(1)\in\mathcal E(\lambda)$, there exists $t\in(0,1)$ so that $c(t)\in S_\lambda$. 
\end{lemma}
\begin{proof}
For small $\lambda$ the map $\phi_\lambda(\omega)$ is injective as $|D_2 f|\leq C$. Indeed, $\phi_\lambda(\omega)=\phi_\lambda(\bar\omega)$ implies 
$$\lambda|\omega-\bar\omega|=\lambda^2|f[\lambda,\omega]-f[\lambda,\bar\omega]|\leq C\lambda^2 |\omega-\bar\omega|.$$
We denote the reflection along $x_3=0$ by $R(x,y,z):=(x,y,-z)$ and define 
$$\bar\phi_\lambda:\Sp^2\rightarrow\R^3,\ \bar\phi(\omega):=\left\{
\begin{aligned}
\phi(\omega) & \textrm{ for $\omega\in\Sp^2_+$,}\\
R(\phi(R(\omega))) & \textrm{ for $\omega\not\in\Sp^2_+$.}
\end{aligned}
\right.
$$
Clearly $\bar\phi$ is injective and continuous. Since $\Sp^2$ is compact, $\bar\phi$ is a homeomorphism onto its image. The Jordan-Brouwer separation theorem (see e.g. \cite{MayerBook}, Corollary 5.24) implies that $\bar\phi_\lambda( \Sp^2)$ separates $\R^3$ into a bounded interior and an unbounded exterior domain. We denote these by $\operatorname{int}(\bar\phi_\lambda)$ and $\operatorname{ext}(\bar\phi_\lambda)$ respectively. The lemma follows by putting
$$\mathcal I_\lambda:=\operatorname{int}(\bar\phi_\lambda)\cap[ z\geq 0]\hspace{.5cm}\textrm{and}\hspace{.5cm} \mathcal E_\lambda:=\operatorname{ext}(\bar\phi_\lambda)\cap[ z\geq 0].$$
The intersection point of $c$ with $S_\lambda$ is a direct consequence form the Jordan-Brouwer separation theorem.
\end{proof}

\begin{lemma}\label{lemma1}
Let $v\in[0,1)$, $\phi_\lambda$ be as in Equation \eqref{abstractsurfaces}, $\theta_0\in\Sp^2_+$ and consider the ray $\R_0^+\theta_0$. There exists $\lambda_0(v)>0$  such that for all $\lambda\in(0,\lambda_0]$ there exists a unique intersection point $t(\lambda,\theta_0)\theta_0\in S_\lambda\cap\R_0^ +\theta_0$. By Lemma \ref{foliationLemma00}, there then exists a unique $\omega(\lambda,\theta_0)\in\Sp^2_+$ such that $t(\lambda,\theta_0)\theta_0=f_\lambda(\omega(\lambda,\theta_0))$. The map $(\lambda,\theta_0)\mapsto (t(\lambda,\theta_0),\omega(\lambda,\theta_0))$ is smooth. 
\end{lemma}
 \begin{proof}
 The existence of one intersection point follows from Lemma \ref{foliationLemma00}. Next we derive uniqueness. Suppose
 \begin{equation}\label{eq1eq2}\\
      t\theta_0= \lambda ve_1+ \lambda\omega+\lambda^2f[\lambda,\omega]\hspace{.5cm}\textrm{and}\hspace{.5cm}
      \bar t\theta_0 =\lambda ve_1+\lambda\bar\omega+\lambda^2f[\lambda,\bar\omega].
  \end{equation}
 Since $f$ is smooth, we have $|f|+|D_2 f|\leq C$ and hence for $\lambda_0$ small enough
  \begin{equation}\label{lowerbound}
  |t|\leq C\lambda,\ |\bar t|\leq C\lambda\hspace{.5cm}\textrm{and}\hspace{.5cm}|t-\bar t|\geq\lambda|\omega-\bar\omega|-\lambda^2|f[\lambda,\omega]-f[\lambda,\bar\omega]|\geq\frac12\lambda|\omega-\bar\omega|.
  \end{equation}
Equation (\ref{eq1eq2}) gives
  \begin{align}
      \lambda^2&=(t\theta_0-\lambda ve_1-\lambda^2 f[\lambda,\omega])^2=t^2-2t(\lambda ve_1+ \lambda^2 f[\lambda,\omega])\cdot\theta_0+\lambda^2 v^2 +\lambda^3 R[\lambda, \omega],\label{eq3}\\
      \lambda^2&=(\bar t\theta_0-\lambda ve_1-\lambda^2 f[\lambda,\bar \omega])^2=\bar t^2-2\bar t(\lambda ve_1+ \lambda^2 f[\lambda,\bar \omega])\cdot\theta_0+\lambda^2 v^2 +\lambda^3 R[\lambda, \bar \omega].\label{eq4}
  \end{align}
Here $R$ is a smooth function in $\lambda$ and $\omega$. Equation \eqref{eq3} is a quadratic equation for $t$. Using that $t\geq 0$ and $v\in[0,1)$, we deduce that for $\lambda$ small enough there is $C(v)\in(0,1)$ such that 
\begin{align}
t-\lambda ve_1\cdot\theta_0=& \lambda^2 f[\lambda,\omega]+\sqrt{(\lambda ve_1\cdot\theta_0 +\lambda^2 f[\lambda,\omega])^2+\lambda^2(1-v^2-\lambda R[\lambda,\omega])}\nonumber\\
\geq&\lambda^2 f[\lambda,\omega]+\lambda\sqrt{1-v^2-\lambda R[\lambda,\omega]}\nonumber\\
\geq&\lambda \sqrt{1-v^2-C\lambda}-C\lambda^2\nonumber\\
\geq &C(v)\lambda.\label{FoliationproofTformula}
\end{align}
  A respective estimate can be derived for $\bar t$. 
  Subtracting (\ref{eq3}) and (\ref{eq4}) and using Estimate \eqref{FoliationproofTformula} gives 
    \begin{align}
    &\left|2t \lambda^2 f[\lambda,\omega]-2\bar t \lambda^2 f[\lambda,\bar\omega]+\lambda^3 R[\lambda,\bar\omega]-\lambda^3 R[\lambda,\omega]\right|\nonumber\\
    \geq &
    |t-\bar t|\ \left|t+\bar t-2\lambda ve_1\cdot\theta_0\right|\nonumber\\
    \geq& C(v)\lambda|t-\bar t|.\label{eqhalf43}
    \end{align}
 Estimating the top line of Equation \eqref{eqhalf43} by adding and subtracting $2t\lambda^2 f[\lambda,\bar\omega]$ yields
  $$C(v)\lambda|t-\bar t|\leq C\lambda^2 |t-\bar t|+C |t|\lambda^2|\omega-\bar\omega|+C\lambda ^3|\omega-\bar\omega|.$$
 For $\lambda$ small enough we can absorb $\lambda^2|t-\bar t|$ to the left. Using that $|t|\leq C\lambda$ we then get $\lambda|t-\bar t|\leq C\lambda^3|\omega-\bar\omega|$.  Combined with Equation (\ref{lowerbound}) we deduce 
  $\frac12\lambda|\bar\omega-\omega|\leq C\lambda^2|\omega-\bar\omega|$ which shows $\omega=\bar\omega$ for small enough $\lambda_0$. Since $|t-\bar t|\leq C\lambda^2|\omega-\bar\omega|$ we also get $t=\bar t$.\\ 
  
  Now we show that $(t,\omega)$ depend regularly on $\lambda$. Let $\lambda_*>0$, $\theta_{0*}\in\Sp^2_+$ and $(t_*,\omega_*)$ be the unique solution to 
  \begin{equation}\label{theatatsmoothdependenceeq01}
  t\theta_{0*}-\lambda_* ve_1-\lambda_*\omega-\lambda_*^2f[\lambda_*,\omega]=0.
  \end{equation}
  We show that $\theta_{0*}\not\in T_{\omega_*}\Sp^2_+$ by demonstrating $\theta_{0*}\cdot\omega_*\neq 0$. Indeed, multiplying Equation \eqref{theatatsmoothdependenceeq01} with $\omega_*$ shows that for $\lambda_*$ sufficiently small
  $$t_*\theta_{0*}\cdot\omega_*=\lambda_*+\lambda_* ve_1\cdot\omega_*+\lambda_*^2f[\lambda_*,\omega_*]\cdot\omega_*\geq \lambda_*(1-v)-C\lambda_*^2>0.$$
  We consider the smooth map
  $$\mathcal G:[0,\infty)\times\Sp^2_+\times [0,\lambda_0)\times \Sp^2_+\rightarrow\R^3,\ \mathcal G(t,\omega,\lambda,\theta_0):=t\theta_0-\lambda ve_1-\lambda\omega-\lambda^2f[\lambda,\omega].$$
  Then clearly $\mathcal G(t_*, \omega_*,\lambda_*,\theta_{0*})=0$. We now prove that $D_{1,2}\mathcal G[t_*,0,\lambda_*,\theta_{0*}]$ is invertible. To do so, we compute 
  \begin{align*} 
  D_{1,2}\mathcal G[t_*, 0,\lambda_*,\theta_{0*}](\lambda_*s,\xi)=&\lambda_*s\theta_{0*}-\lambda_*\xi-\lambda_*^2D_2 f[\lambda_*,\omega_*]\xi\\
  =&\lambda_*\left(s\theta_{0*}-\xi-\lambda_*D_2 f[\lambda_*, \omega_*]\xi\right).
  \end{align*}
  The map $\R\times T_{\omega_*}\Sp^2_+\ni (s,\xi)\mapsto s\theta_{0*}-\xi\in\R^3$ is an invertible linear map since $\theta_0\not\in T_{\omega_*} \Sp^2_+$. So, for small enough $\lambda_*$ the operator $D_{1,2}\mathcal G[t_*, 0,\lambda_*]$ is also invertible. By the implicit function theorem, the map $(\lambda,\theta_0)\mapsto (\omega(\lambda,\theta_0),t(\lambda,\theta_0))$ is locally and therefore also globally smooth.  
 \end{proof}

\begin{lemma}[Foliation case]\label{follem1}
Let $v\in[0,1)$ and $\phi_\lambda$ be as in Equation \eqref{abstractsurfaces}. There exists $\lambda_0(v)>0$ such that $\mathcal I_{\lambda_0}\backslash\set 0$ is foliated by $S_\lambda$ with $\lambda\in(0,\lambda_0)$. 
\end{lemma}

\begin{proof}
First, we claim that for $v\in[0,1)$ there exists $\lambda_0(v)>0$ such that for all $0<\lambda_1<\lambda_2 \leq \lambda_0$ we have $\operatorname{im}\phi_{\lambda_1}\cap\operatorname{im}\phi_{\lambda_2}=\emptyset$.
Indeed suppose that $\omega,\bar\omega\in\Sp^2_+$ and $\epsilon>0$ 
\begin{equation}\label{eq}
\lambda v e_1+\lambda\omega+\lambda^2f[\lambda,\omega]=(\lambda+\epsilon) v e_1 +(\lambda+\epsilon)\bar\omega+(\lambda+\epsilon)^2f[\lambda+\epsilon,\bar\omega].
\end{equation}
Using the smoothness of $f$, we deduce 
\begin{align}
    \lambda&=|\lambda\omega|\nonumber\\
    &\geq |\epsilon ve_1+(\lambda+\epsilon)\bar\omega|-|(\lambda+\epsilon)^2f[\lambda+\epsilon,\bar\omega]-\lambda^2f[\lambda,\omega]|\nonumber\\
    &\geq \lambda+\epsilon-\epsilon v-C|(\lambda+\epsilon)^2-\lambda^2|-C\lambda^2 |\omega-\bar\omega|\nonumber\\
    &\geq \lambda+\epsilon-\epsilon v-C(\lambda\epsilon+\epsilon^2)-C\lambda^2 |\omega-\bar\omega|\nonumber\\
    &\geq \lambda+\epsilon(1-v-C\lambda-C\epsilon)-C\lambda^2 |\omega-\bar\omega|.\label{FoliationCaseeq}
\end{align}
Next, we use Equation \eqref{eq} to estimate 
\begin{align*} 
\lambda |\omega-\bar\omega|&\leq (1+v)\epsilon+|(\lambda+\epsilon)^2f[\lambda+\epsilon,\bar\omega]-\lambda^2f[\lambda,\omega]|\\
&\leq (1+v)\epsilon+|(\lambda+\epsilon)^2-\lambda^2| \ |f[\lambda+\epsilon,\bar\omega]|+\lambda^2 |f[\lambda+\epsilon,\bar\omega]-f[\lambda,\omega]|\\
&\leq (1+v)\epsilon+C(\lambda\epsilon+\epsilon^2)+C\lambda^2(|\omega-\bar\omega|+\epsilon)\\
&\leq (1+v)\epsilon+C(\lambda\epsilon+\epsilon^2)+C\lambda^2|\omega-\bar\omega|.
\end{align*}
Choosing $\lambda_0$ small enough we can absorb $\lambda^2|\omega-\bar\omega|$ to the left to get $\lambda|\omega-\bar\omega|\leq C\epsilon$ and hence $\lambda^2 |\omega-\bar\omega|\leq C\lambda\epsilon$. Inserting into Estimate \eqref{FoliationCaseeq} gives
$$\lambda\geq\lambda+\epsilon(1-v-C\lambda-C\epsilon).$$
We have the estimates $\lambda\leq\lambda_0$ and $\epsilon\leq\lambda+\epsilon\leq \lambda_0$. Since $v\in[0,1)$ we can choose $\lambda_0$ so small that $1-v-2C\lambda_0>c_1(v)>0$. We then obtain $\lambda\geq \lambda+\epsilon c_1$ and hence $\epsilon=0$,
which implies the claim.\\

From now on let $\lambda_0(v)>0$ be so small such that the surfaces $\phi_\lambda$ are disjoint for $\lambda<\lambda_0$ and such that Lemma \ref{lemma1} is applicable. We claim that $\mathcal I_{\lambda_0}\backslash\set 0=\bigcup_{\lambda<\lambda_0}\operatorname{im}\phi_\lambda$.
 To prove this consider a point $0\neq p\in \mathcal I_{\lambda_0}$ and define $\theta_0:=\frac p{|p|}$. Lemma \ref{lemma1} equips us with a map $t:(0,\lambda_0]\rightarrow\R_0^ +$ such that $\R_0^ +\theta_0\cap S_\lambda=\set{t(\lambda)\theta_0}$.\\
 
 $t(\lambda_0)>|p|$. To see this, we first note $t\theta_0\in\mathcal E_\lambda$ for $t\rightarrow\infty$. Thus, $t\theta_0 \in \mathcal E_\lambda$ for all $t\geq t(\lambda_0)$. Now, if $t(\lambda_0)\leq |p|$ was true we would get $p=|p|\theta_0\in\mathcal E_\lambda$ but $p\in \mathcal I_{\lambda_0}$. 
 Since $S_\lambda$ shrinks to the origin as $\lambda\rightarrow 0^+$ we also deduce that $t(\lambda)\rightarrow 0$ as $\lambda\rightarrow 0^+$. The continuity of the map $t$ then gives $\lambda\in(0, \lambda_0)$ such that $t(\lambda)=|p|$. So $p=t(\lambda)\theta_0\in S_\lambda=\operatorname{im}(\phi_\lambda)$ for some $\lambda\in(0,\lambda_0)$.\\
 
We have shown that for all $p\in \mathcal I_{\lambda_0}\backslash\set 0$ we have a unique $\lambda\in(0,\lambda_0)$ such that $p\in S_\lambda$. In other words, the map 
$$\mathcal F:(0,\lambda_0)\times\Sp^2_+\rightarrow \mathcal I_{\lambda_0}\backslash\set 0,\ (\lambda,\omega)\mapsto \phi_\lambda(\omega)$$
is smooth and bijective. We claim that $\mathcal F$ is a diffeomorphism. Once this is shown, any atlas of $\Sp^2_+$ provides a foliated atlas of $\mathcal I_{\lambda_0}\backslash\set 0$. First we note that for any $\omega\in \Sp^2_+$ the vector $\omega+ve_1\not\in T_\omega\Sp^2_+$ as $\omega\cdot (\omega+ve_1)\geq 1-|v|>0$. So, the linear map 
$$\R\times T_\omega\Sp^2_+\rightarrow\R^3,\ (s, \xi)\mapsto s(\omega+ve_1)+ \xi\hspace{.5cm}\textrm{is a linear isomorphism}.$$
We compute
\begin{align*}
D\mathcal F(\lambda,\omega)(\lambda s,\xi)=&\lambda s(ve_1+\omega)+\lambda\xi +\lambda^2 D_2 f[\lambda,\omega]\xi+2\lambda^2 s f[\lambda,\omega]+\lambda^3 D_1 f[\lambda,\xi]s\\
=&\lambda \left[ s(ve_1+\omega)+\xi +\lambda \left(D_2 f[\lambda,\omega]\xi+2 s f[\lambda,\omega]+\lambda^2 D_1 f[\lambda,\xi]s\right)\right].
\end{align*}
Consequently, for $\lambda>0$ small enough, $D\mathcal F(\lambda,\omega):\R\times T_\omega\Sp^2_+\rightarrow \R^3$ is an isomorphism. So $\mathcal F$ is a local and, therefore, by bijectivity of $\mathcal F$, also a global diffeomorphism. 
 \end{proof}

Finally, we prove that for $v>1$, the surfaces $\phi_\lambda$ cannot provide a foliation of any deleted neighborhood of $0$. 

\begin{lemma}\label{follem2}
Let $v>1$ and $\phi_\lambda$ be as in Equation \eqref{abstractsurfaces}. Then, for any $\lambda_0>0$ there exist $\lambda_1>0$ and $\epsilon>0$ such that $0<\lambda_1<\lambda_1+\epsilon<\lambda_0$ and $\operatorname{im}\phi_{\lambda_1}\cap\operatorname{im}\phi_{\lambda_1+\epsilon}\neq\emptyset$. In particular, for any $\lambda_0>0$ the surfaces $(S_\lambda)_{\lambda\in(0,\lambda_0)}$ cannot form a foliation. 
\end{lemma}
The following proof is best understood by recalling the `unperturbed' case we considered in Figure \ref{figure}.
\begin{proof}
Let $\lambda_0>0$ and $\lambda_1<\frac13\lambda_0$ to be chosen later. We put $\epsilon:=\frac{\lambda_1}{v-1}$ and assume that $\lambda_1(v)$ is chosen so small that $\lambda_1+\epsilon<\lambda_0$. For $\vartheta\in[0,1]$ we consider the curve $c_\vartheta:\partial \Sp^2_+\rightarrow\R^2$ defined by
$c_\vartheta(\omega):= \lambda_1 ve_1+\lambda_1\omega+f[\vartheta\lambda_1,\omega]$. Just as in the proof of Lemma \ref{foliationLemma00}, we see that $c_\vartheta$ is injective for small enough $\lambda_1$ and hence separates $\R^2$ into an interior domain $U(\vartheta)$ and an exterior domain $\tilde U(\vartheta)$. Letting
$$p_\pm(\lambda):=\phi_\lambda(\pm e_1)=\lambda ve_1\pm \lambda e_1+\lambda^2 f[\lambda,\pm e_1],$$
we wish to prove that for small enough $\lambda_1$ we get $p_-(\lambda_1+\epsilon)\in U(1)$.  In a first step we show $p_-(\lambda_1+\epsilon)\in U(0)=\lambda_1 ve_1+\lambda D_1(0)$ where $D_1(0):=\set{x\in\R^2\ | |x|<1}$. Indeed, for $\lambda_1(v)$ small enough  and recalling the definition of $\epsilon$, we compute 
\begin{align*} 
|p_-(\lambda_1+\epsilon)-\lambda_1 v e_1|&=|\epsilon ve_1-(\lambda_1+\epsilon)e_1+(\lambda_1+\epsilon)^2f[\lambda_1+\epsilon,-e_1]|\\
&\leq |\lambda_1+\epsilon-\epsilon v|+C(\lambda_1+\epsilon)^2\\
&\leq C(v)\lambda_1^2\\
&<\lambda_1.
\end{align*}
Next we prove that the distance of $p_-(\lambda_1+\epsilon)$ to the boundary $\partial U(\vartheta)=\operatorname{im}(c_\vartheta)$ is bounded from below uniformly for all $\vartheta\in[0,1]$. It then follows that $p_-(\lambda_1+\epsilon)\in U(1)$. Indeed, for any $\omega\in\partial\Sp^2_+$ and small enough $\lambda_1$ we estimate
\begin{align}
    &|c_\vartheta(\omega)-p_-(\lambda_1+\epsilon)|\nonumber\\
    =&|\lambda_1 ve_1+\lambda_1\omega +\vartheta^2\lambda_1^2f[\vartheta\lambda_1,\omega]-(\lambda_1+\epsilon) ve_1+(\lambda_1+\epsilon)e_1-(\lambda_1+\epsilon)^2f[\lambda_1+\epsilon,-e_1]|\nonumber\\
    \geq &|\lambda_1\omega -\epsilon ve_1 +(\lambda_1+\epsilon)e_1|-C(v)\lambda_1^2\nonumber\\
    \overset{(!)}=&|\lambda_1\omega|-C(v)\lambda_1^2\nonumber\\
    \geq & \frac{\lambda_1} 2.\label{Joradancurve_Positivedistance}
\end{align}
To verify $(!)$, we have used the definition of $\epsilon$. Assuming $p_-(\lambda_1+\epsilon)\not\in U(1)$, there would exist $\vartheta\in(0,1]$ such that $p_-(\lambda_1+\epsilon)\in \operatorname{im}(c_\vartheta)$, which is impossible by Estimate \eqref{Joradancurve_Positivedistance}. A similar argument shows $p_+(\lambda_1+\epsilon)\in\tilde U(1)$. 
Now let $\gamma:[0,1]\rightarrow\Sp^2_+$ be any continuous curve satisfying $\gamma(0)=-e_1$ and $\gamma(1)=e_1$ and consider $\Gamma(t):=\phi_{\lambda_1+\epsilon}(\gamma(t))$.
Then $\Gamma(0)=p_-(\lambda_1+\epsilon)\in U(0)\subset \mathcal I_{\lambda_1}$ and $\Gamma(1)=p_+(\lambda_1+\epsilon)\in \tilde U(1)\subset \mathcal E_{\lambda_1}$. By Lemma \ref{foliationLemma00} there exists $t_0\in(0,1)$ satisfying $\Gamma(t_0)\in S_{\lambda_1}$. Of course, we also have $\Gamma(t_0)\in S_{\lambda_1+\epsilon}$, which implies the Lemma.
\end{proof}
\setcounter{equation}{0}
\section{Computations}
In the following, we will often refer to machine computations. All of these have been executed using Mathematica \cite{mathematica} (see Appendix \ref{mathematicacodeAppendix}). We provide the commented code as supplementary material.

\subsection{General Formulas}\label{Formulas}
In this subsection, let $\phi:U\subset\R^2\rightarrow\Sp^2_+$ be a local parameterization. We consider a function $\varphi:\Sp^2_+\rightarrow\R$ that is usually in $C^{2,\alpha}(\Sp^2_+)$ but occasionally, for the Willmore case, even in $C^{4,\alpha}(\Sp^2_+)$ and put $f_\varphi:\Sp^2_+\rightarrow\R^3,\ f_\varphi(\omega):=(1+\varphi(\omega))\omega$. Let $q$ denote a smooth, symmetric matrix on $\R^3$ so that $\tilde g:=\delta+\epsilon q$  is a
background metric on $\R^3$ that is close to the euclidean metric $\delta$ when $\epsilon$ is small.  

\paragraph{First Fundamental Form}\ \\
The first fundamental form of $f_\varphi$ with respect to the background metric $\tilde g$ is given by
$$g_{ij}=g_{ij}[\varphi,\tilde g]=\tilde g((1+\varphi)\partial_i\phi+\partial_i\varphi\phi, (1+\varphi)\partial_j\phi+\partial_j\varphi\phi).$$
This gives the following two formulas:
\begin{align}
&D_1 g_{ij}[0,\delta]\varphi=2\varphi\langle\partial_i\phi,\partial_j\phi\rangle\label{FirstFFvariationPHI}\\
&D_1^2 g_{ij}[0,\delta](\varphi,\varphi)=2\left( \varphi^2\langle\partial_i\phi,\partial_j\phi\rangle+\partial_i\varphi\partial_j\varphi\right)\label{FirstFFSecondVariationPHI}\\
&D_2 g_{ij}[0,\delta]q=q(\partial_i\phi,\partial_j\phi)\label{SecondFFvariationPHI}
\end{align}
The induced surface measure $d\mu[\varphi,\tilde g]=\sqrt{\det g}d^2x$ then satisfies
\begin{align}
&D_1 d\mu[0,\delta]\varphi=2\varphi d\mu_{\Sp^2},\label{GrammVarPhi}\\
&D_1^2 d\mu[0,\delta](\varphi,\varphi)=\left(|\nabla \varphi|^2+2\varphi^2\right) d\mu_{\Sp^2},\label{GrammVarPhiSecond}\\
&D_2 d\mu[0,\delta]q=\frac12 g^{ij}q(\partial_i\phi,\partial_j\phi)d\mu_{\Sp^2}=\frac12\tr_{\Sp^2} q d\mu_{\Sp^2}.\label{GrammVarMet}
\end{align}

\noindent
\textit{Proof of Equation \eqref{GrammVarPhiSecond}}\ \\
We put $g:=g[\epsilon\varphi,\delta]$ and $g_0:=g[0,\delta]$ and compute 
\begin{align}
    \frac {d^2}{d\epsilon^2}\bigg|_{\epsilon=0}d\mu_g&=\frac {d^2}{d\epsilon^2}\bigg|_{\epsilon=0}\sqrt{\det g}d^2x\nonumber\\
    &=\frac {d}{d\epsilon}\bigg|_{\epsilon=0}\left[\frac1{2\sqrt{\det g}}\frac d{d\epsilon}\det g d^2x\right]\nonumber\\
    &=-\frac1{4\sqrt{\det g_0}^3}\left(\frac d{d\epsilon}\bigg|_{\epsilon=0}\det g\right)^2d^2x+\frac1{2\sqrt{\det g_0}}\frac {d^2}{d\epsilon^2}\bigg|_{\epsilon=0}\det gd^2x.\label{SecondMeasureVarioationproofeq01}
\end{align}
Using Equation \eqref{FirstFFvariationPHI}, we first compute
\begin{equation}\label{SecondMeasureVarioationproofeq02}
\frac d{d\epsilon}\bigg|_{\epsilon=0}\det g=\det (g_0)g_0^{ij}\frac{\partial g_{ij}}{\partial\epsilon}\bigg|_{\epsilon=0}
=\det (g_0)g_0^{ij}2\varphi\cdot  (g_0)_{ij}=4\det (g_0) \varphi.
\end{equation}
By Equation \eqref{FirstFFvariationPHI} we have $g=(1+2\epsilon\varphi)g_0+\mathcal O(\epsilon^2)$. Using also Equation \eqref{FirstFFSecondVariationPHI}, we get
\begin{align}
\frac{d^2}{d\epsilon^2}\bigg|_{\epsilon=0}\det g&=\det(g_0)g_0^{ij}\frac{\partial^2 g_{ij}}{\partial\epsilon^2}\bigg|_{\epsilon=0}+\frac{d^2}{d\epsilon^2}\bigg|_{\epsilon=0}\det\left[g_0+\epsilon \frac{\partial g}{\partial \epsilon}\bigg|_{\epsilon=0}\right]\nonumber\\
&=\det(g_0)g_0^{ij}2\left(\varphi^2\cdot (g_0)_{ij}+\partial_i\varphi\partial_j\varphi\right)+\frac{d^2}{d\epsilon^2}\bigg|_{\epsilon=0}\det((1+2\epsilon\varphi)g_0)\nonumber\\
&=\det(g_0)\left(4\varphi^2+2|\nabla\varphi|^2\right)+\frac{d^2}{d\epsilon^2}\bigg|_{\epsilon=0}(1+2\epsilon\varphi)^2\det(g_0)\nonumber\\
&=\det(g_0)\left(12\varphi^2+2|\nabla\varphi|^2\right).\label{SecondMeasureVarioationproofeq03}
\end{align}
Putting Equations \eqref{SecondMeasureVarioationproofeq02} and \eqref{SecondMeasureVarioationproofeq03} into Equation \eqref{SecondMeasureVarioationproofeq01}, we deduce Equations \eqref{GrammVarPhiSecond}.\qed

\paragraph{Normal}\ \\
Let $\tilde \nu[\varphi,\tilde g]$ denote the interior unit normal along $f_\varphi$ with respect to the background metric $\tilde g$. The following formulas hold:
\begin{align}
    &D_1 \tilde\nu[0,\delta]\varphi=g^{ij}\partial_i\varphi\partial_j\phi\label{NormalVariationPHI}\\
    &\langle D_1^2\tilde\nu[0,\delta](\varphi,\varphi),\omega\rangle=|\nabla\varphi|^2\label{NormalSecondVariationPHI}\\
    &D_2 \tilde\nu[0,\delta]q=g^{ij}q(\phi,\partial_i\phi)\partial_j\phi+\frac12 q(\phi,\phi)\phi\label{NormalVariationGeneralMET}
\end{align}
For a proof of Equation \eqref{NormalVariationGeneralMET} we refer to the proof of Lemma 7 in \cite{AK}.\\

\noindent
\textit{Proof of Equations \eqref{NormalVariationPHI} and \eqref{NormalSecondVariationPHI}}\ \\
To prove Equation \eqref{NormalVariationPHI}, we note that 
$$\langle\tilde \nu[\epsilon\varphi,\delta],\tilde\nu[\epsilon\varphi,\delta]\rangle\equiv 1\hspace{.5cm}\textrm{and so}\hspace{.5cm}\langle\tilde\nu[0,\delta], \frac{\partial}{\partial\epsilon}\bigg|_{\epsilon=0}\tilde\nu[\epsilon\varphi,\delta]\rangle=0.$$
Hence $D_1\tilde\nu[0,\delta]\varphi$ is tangential. Recall that $\tilde\nu[0,\delta]=-\omega=-\phi$. We have 
$$\langle\tilde \nu[\epsilon\varphi,\delta],\partial_i f_{\epsilon\varphi}\rangle\equiv 0\hspace{.5cm}\textrm{and so}\hspace{.5cm}\langle\frac{\partial}{\partial\epsilon}\bigg|_{\epsilon=0}\tilde\nu[\epsilon\varphi,\delta], \partial_i\phi\rangle=\langle\phi, \partial_i\varphi\phi+\varphi\partial_i\phi\rangle=\partial_i\varphi.$$
This is Equation \eqref{NormalVariationPHI}. To get Equation \eqref{NormalSecondVariationPHI}, we differentiate the identity $\langle\tilde\nu[\epsilon\varphi,\delta],\tilde\nu[\epsilon\varphi,\delta]\rangle\equiv~1$ twice with respect to $\epsilon$ and evaluate at $\epsilon=0$. Using $\nu[0,\delta]=-\phi$  and Equation \eqref{NormalVariationPHI}, we get
\begin{align*} 
2\langle\frac{\partial^2}{\partial\epsilon^2}\bigg|_{\epsilon=0}\tilde\nu[\epsilon\varphi,\delta], \phi\rangle
=&2\langle\frac{\partial}{\partial\epsilon}\bigg|_{\epsilon=0}\tilde\nu[\epsilon\varphi,\delta],\frac{\partial}{\partial\epsilon}\bigg|_{\epsilon=0}\tilde\nu[\epsilon\varphi,\delta]\rangle\\
=&2g^{ab} g^{kl}\partial_a\varphi\partial_k\varphi\langle\partial_b\phi,\partial_l\phi\rangle\\
=&g^{ak}\partial_a\varphi\partial_l\varphi\\
=&2|\nabla\varphi|^2.&\qed
\end{align*}

\paragraph{Mean curvature}\ \\ 
The mean curvature satisfies
\begin{align} 
D_1 H[0,\delta]\varphi&=-(\Delta+2)\varphi,\label{HlinearizationPHI}\\
D_1^2 H[0,\delta](\varphi,\varphi)&=4\varphi\Delta\varphi+4\varphi^2,\label{HSecondVariationPHI}\\
D_2 H[0,\delta]q&=\frac12\tr_{\Sp^2}\nabla_{\ \omega} q-\tr_{\Sp^2}\nabla_\cdot q(\omega ,\cdot)+q(\omega,\omega)-\tr_{\Sp^2}q.\label{HlinearizationGeneralMET}
\end{align}
Equation \eqref{HlinearizationGeneralMET} is to be understood as $(D_2 H[0,\delta] q)(\omega)=...$ and is established in the proof of Lemma 7 in \cite{AK}.\\

\noindent
\textit{Proof of Equations \eqref{HlinearizationPHI} and \eqref{HSecondVariationPHI}}\ \\
We assume that $\phi$ is an almost euclidean parameterization of $\Sp^2_+$. That is, it satisfies $g_{ij}(0)=\langle\partial_i\phi(0),\partial_j\phi(0)\rangle=\delta_{ij}$ and $\partial_i g_{ab}(0)=0$. It follows that $\langle\partial_{ij}\phi,\partial_b\phi\rangle(0)=0$ and hence 
$$\partial_{ij}\phi(0)=\langle\partial_{ij}\phi,\phi\rangle(0)\phi(0)=-\langle\partial_i\phi,\partial_j\phi\rangle(0)\phi(0)=-\delta_{ij}\phi(0).$$
In all of the following computations, we will suppress `$(0)$' in the notation. Using Equation \eqref{FirstFFvariationPHI}, we compute
\begin{equation}\label{HLINEq1}
\frac{\partial g^{ij}[\epsilon\varphi,\delta]}{\partial\epsilon}\bigg|_{\epsilon=0}=-g^{ia}[0,\delta]g^{jb}[0,\delta]\frac{\partial g_{ab}[\epsilon\varphi,\delta]}{\partial\epsilon}\bigg|_{\epsilon=0}=-2\varphi \delta_{ij}.
\end{equation}
Next, using $\partial_{ij}\phi=-\delta_{ij}\phi$, $\langle\phi,\partial_i\phi\rangle=0$, $\tilde\nu[0,\delta]=-\phi$ and Equation \eqref{NormalVariationPHI} we compute
\begin{align} 
\frac{\partial}{\partial\epsilon}\bigg|_{\epsilon=0}\langle\partial_{ij}[(1+\epsilon\varphi)\phi],\tilde\nu[\epsilon\varphi,\delta]\rangle=&\langle\partial_{ij}\varphi\phi+\varphi\partial_{ij}\phi+\partial_i\varphi\partial_j\phi+\partial_j\varphi\partial_i\phi, -\phi\rangle+\langle\partial_{ij}\phi, \partial_a\varphi\partial_a\phi\rangle\nonumber\\
%-------
=&-\partial_{ij}\varphi+\delta_{ij}\varphi.\label{HLINEq2}
\end{align}
Due to the almost euclidean coordinates, we have $\partial_{ii}\varphi=\Delta\varphi$. Combining this with Equations \eqref{HLINEq1} and \eqref{HLINEq2}, we compute 
\begin{align*}
    \frac{\partial}{\partial\epsilon}\bigg|_{\epsilon=0}H[\epsilon\varphi,\delta]=& \frac{\partial}{\partial\epsilon}\bigg|_{\epsilon=0}\bigg [g^{ij}[\epsilon\varphi,\delta]\langle\partial_{ij}[(1+\epsilon\varphi)\phi],\tilde\nu[\epsilon\varphi,\delta]\rangle\bigg]\\
    =&\frac{\partial g^{ij}[\epsilon\varphi,\delta]}{\partial\epsilon}\bigg|_{\epsilon=0}\langle\partial_{ij}\phi,-\phi\rangle+\delta_{ij}\left(-\partial_{ij}\varphi+\delta_{ij}\varphi\right)\\
    =&-2\varphi\delta_{ij}\delta_{ij}-\partial_{ii}\varphi+2\varphi\\
    =&-(\Delta+2)\varphi.
\end{align*}
Thus Equation \eqref{HlinearizationPHI} is established. To establish Equation \eqref{HSecondVariationPHI}, we compute
\begin{align}
    \frac{\partial^2}{\partial\epsilon^2}\bigg|_{\epsilon=0} H[\epsilon\varphi,\delta]=&\frac{\partial^2 g^{ij}[\epsilon\varphi,\delta]}{\partial\epsilon^2}\bigg|_{\epsilon=0}\langle\partial_{ij}\phi,-\phi\rangle
    +
    \delta_{ij}\frac{\partial^2}{\partial\epsilon^2}\bigg|_{\epsilon=0}\langle\partial_{ij}\left[(1+\epsilon\varphi)\phi\right],\tilde\nu[\epsilon\varphi,\delta]\rangle\nonumber\\
    &+2\frac{\partial g^{ij}[\epsilon\varphi,\delta]}{\partial\epsilon}\bigg|_{\epsilon=0}\frac{\partial}{\partial\epsilon}\bigg|_{\epsilon=0}\langle\partial_{ij}\left[(1+\epsilon\varphi)\phi\right],\tilde\nu[\epsilon\varphi,\delta]\rangle.\label{HLINEq3}
\end{align}
First, combining Equations \eqref{HLINEq1} and \eqref{HLINEq2} we get 
\begin{equation}\label{HLINEq4}
    2\frac{\partial g^{ij}[\epsilon\varphi,\delta]}{\partial\epsilon}\bigg|_{\epsilon=0}\frac{\partial}{\partial\epsilon}\bigg|_{\epsilon=0}\langle\partial_{ij}\left[(1+\epsilon\varphi)\phi\right],\tilde\nu[\epsilon\varphi,\delta]\rangle
    =
    -4\varphi\delta_{ij}(\delta_{ij}\varphi-\partial_{ij}\varphi)
    =
    -8\varphi^2+4\varphi\Delta\varphi.
\end{equation}
We note that $\langle\partial_{ij}\phi,-\phi\rangle=\delta_{ij}$ and that $g_{ij}[0,\delta]=\delta_{ij}$. Using these relations and Equations \eqref{FirstFFvariationPHI} and \eqref{FirstFFSecondVariationPHI} we write
\begin{align}
    \frac{\partial^2 g^{ij}[\epsilon\varphi,\delta]}{\partial\epsilon^2}\bigg|_{\epsilon=0}\langle\partial_{ij}\phi,-\phi\rangle
    =&
    \delta_{ij}\left[2\left(\frac{\partial g[\epsilon\varphi,\delta]}{\partial\epsilon}\bigg|_{\epsilon=0}\right)^2_{ij}-\frac{\partial^2g_{ij}[\epsilon\varphi,\delta]}{\partial\epsilon^2}\right]\nonumber\\
    =&\delta_{ij}(8\varphi^2\delta_{ij}-2\varphi^2\delta_{ij}-2\partial_i\varphi\partial_j\varphi)\nonumber\\
    =&12\varphi^2-2|\nabla\varphi|^2.\label{HLINEq5}
\end{align} 
Finally, we note $\partial_\epsilon^2 f_{\epsilon\varphi}=0$, $\langle\phi,\partial_i\phi\rangle=0$ and $\partial_{ij}\phi=-\delta_{ij}\phi$. Using these formulas and Equations \eqref{NormalVariationPHI} and \eqref{NormalSecondVariationPHI}, we compute
\begin{align}
    \delta_{ij}\frac{\partial^2}{\partial\epsilon^2}\bigg|_{\epsilon=0}\langle\partial_{ij}[(1+\epsilon\varphi)\phi],\tilde\nu[\epsilon\varphi,\delta]\rangle=&
    2\langle\partial_{ii}(\varphi\phi),D_1\tilde\nu[0,\delta]\varphi\rangle+\langle\partial_{ii}\phi, D_1^2\tilde\nu[0,\delta](\varphi,\varphi)\rangle\nonumber\\
    %--------------
    =&
    2\langle\partial_{ii}\varphi\phi+\varphi\partial_{ii}\phi+2\partial_i\varphi\partial_i\phi,\partial_a\varphi\partial_a\phi\rangle-\delta_{ii}\langle\phi, D_1^2\tilde\nu[0,\delta](\varphi,\varphi)\rangle\nonumber\\
     %--------------
    =&
    2\langle2\partial_i\varphi\partial_i\phi,\partial_a\varphi\partial_a\phi\rangle-\delta_{ii}\langle\phi, D_1^2\tilde\nu[0,\delta](\varphi,\varphi)\rangle\nonumber\\
    %--------------
    =&
   4|\nabla\varphi|^2-2|\nabla\varphi|^2\nonumber\\
   =&2|\nabla\varphi|^2.\label{HLINEq6}
\end{align}
Inserting Equations \eqref{HLINEq4}, \eqref{HLINEq5} and \eqref{HLINEq6} into Equation \eqref{HLINEq3} we get
$$ \frac{\partial^2}{\partial\epsilon^2}\bigg|_{\epsilon=0} H[\epsilon\varphi,\delta]=12\varphi^2-2|\nabla\varphi|^2+2|\nabla\varphi|^2-8\varphi^2+4\varphi\Delta\varphi=4(\varphi\Delta\varphi+\varphi^2),$$
which is precisely Equation \eqref{HSecondVariationPHI}. \hfill\qed\\

\paragraph{Willmore Energy and Willmore Operator}\ \\
Equations \eqref{GrammVarPhi}, \eqref{GrammVarMet}, \eqref{HlinearizationPHI}, \eqref{HlinearizationGeneralMET} and $H[0,\delta]=2$ imply
\begin{align}
    D_2 \mathcal W[0,\delta]q&=\int_{\Sp^2_+}-\frac12\tr_{\Sp^2}q+q(\omega,\omega)-\tr_{\Sp^2}\nabla_\cdot q(\cdot,\omega)+\frac 12\tr_{\Sp^2}\nabla_\omega q d\mu_{\Sp^2},\label{WillEnDerq}\\
      D_1\mathcal W[0,\delta]\varphi&=-\int_{\Sp^2_+}\Delta \varphi d\mu_{\Sp^2}=\int_{\partial\Sp^2_+}\frac{\partial\varphi}{\partial\eta}dS\label{WillEnDeru}.
\end{align}
The last formula follows from Gauß's theorem and by recalling that $\eta=\tilde\eta[0,\delta]$ is the interior normal along $\partial\Sp^2_+$. Considering the definition of the Willmore operator in Equation \eqref{Willmoreoperatordefinition}, $H[0,\delta]=2$, $h^0[0,\delta]=0$ and $\operatorname{Ric}_\delta=0$, we can use Equations \eqref{HlinearizationPHI} and \eqref{HlinearizationGeneralMET} to get
\begin{align}
D_1 W[0,\delta]\varphi&=-\frac12\Delta(\Delta+2)\varphi,\label{DWphi}\\
D_2 W[0,\delta]q&=\frac12\Delta\left(\frac12\tr_{\Sp^2}\nabla_{\ \omega} q-\tr_{\Sp^2}\nabla_\cdot q(\omega ,\cdot)+q(\omega,\omega)-\tr_{\Sp^2}q\right)+(D \operatorname{Ric}[\delta] q)(\omega,\omega).\label{DWq}
\end{align}
The last term in Equation \eqref{DWq} is not of interest to us since all considered metrics are flat.

\paragraph{Area, Barycenter and Volume}\ \\
Equations \eqref{GrammVarPhi} and \eqref{GrammVarMet} imply
\begin{equation}\label{AreaVariationPHI}
 D_1 A[0,\delta]\varphi=2\int_{\Sp^2_+}\varphi d\mu_{\Sp^2}
 \hspace{.5cm}\textrm{and}\hspace{.5cm}
 D_2 A[0,\delta] q=\frac12\int_{\Sp^2_+}\tr_{\Sp^2} q d\mu_{\Sp^2}.
\end{equation}
For the variation of the barycenter, we refer to Equation (5.4) in \cite{AK}. Here the $L^2$-gradient with respect to the inwards pointing normal is shown to be $-\frac{3}{2\pi}\omega_i$. Since $\frac\partial{\partial\epsilon} f_{\epsilon\varphi}=\varphi\omega=-\varphi\nu[0,\delta]$, this implies 
\begin{equation}\label{BarycenterVariationPHI}
D_1 C^i[0,\delta]\varphi=\frac3{2\pi}\int_{\Sp^2_+}\omega_i\varphi d\mu_{\Sp^2}\hspace{.5cm}\textrm{for $i=1,2$}.
\end{equation}
Denote by $\Omega_\varphi$ the domain in $\R^3_+$ that is bounded by the graph of $f_\varphi$. So in particular $\Omega_0=B_1^+(0)$ Then, the volume functional is defined as
\begin{equation}\label{volumefunctionalDef}
V[\varphi,\tilde g]=\int_{\Omega_\varphi}\sqrt{\det{\tilde g}} d^3x
\hspace{.5cm}\textrm{and in particular}\hspace{.5cm}
V[0,\tilde g]=\int_{B_1^+(0)}\sqrt{\det{\tilde g}}d^3x.
\end{equation}
The following formulas hold:
\begin{align}
    &D_1 V[0,\delta]\varphi=\int_{\Sp^2_+}\varphi d\mu_{\Sp^2}\hspace{.5cm}\textrm{and}\hspace{.5cm}D_1^2V[0,\delta](\varphi,\varphi)=2\int_{\Sp^2_+}\varphi^2 d\mu_{\Sp^2}\label{VolumeVariationPHI}\\
    &D_2 V[0,\delta] q=\frac12\int_{B_1^+(0)}\tr_{\R^3} q d^3x\label{VolumeVariationGeneralMET}
\end{align}
Equation \eqref{VolumeVariationGeneralMET} follows the second formula in Equation \eqref{volumefunctionalDef}. To get the two identities from Equation \eqref{VolumeVariationPHI} we use that for $\|\varphi\|_{C^{2,\alpha}(\Sp^2_+)}$ small enough and $\tilde g=\delta$ we have 
$$V[\varphi,\delta]=\int_{\Sp^2_+}\int_0^{1+\varphi(\omega)}r^2 dr d\mu_{\Sp^2}=\frac13 \int_{\Sp^2_+}(1+\varphi(\omega))^3 d\mu_{\Sp^2}.$$

\paragraph{Boundary Conditions}\ \\
The boundary operator from Equation \eqref{BoundaryOperatorDefinition} is
$$B[\varphi,\tilde g]
=(B_1[\varphi,\tilde g],B_2[\varphi,\tilde g])
:=\left(\omega_3-g^{ij}\tilde g(\omega,\partial_i f_\varphi)\partial_j f_\varphi^3, \frac{\partial H}{\partial\tilde \eta}+H\tilde h^{\R^2}(\tilde \nu,\tilde \nu)\right).$$
We denote by $\eta:=\tilde\eta[0,\delta]$. Using that $\nabla\omega_3=\eta$ along $\partial\Sp^2_+$ and $\langle\omega,\partial_i f_0\rangle=0$, it is readily checked that 
\begin{align}
    D_1 B_1[0,\delta]\varphi&=- g^{ij}\langle \omega,\partial_i \varphi \omega\rangle \partial_j  \phi_3=- g^{ij} \partial_i \varphi\partial_j \phi_3=-\frac{\partial\varphi}{\partial\eta},\label{B1Derphi}\\
    D_2 B_1[0,\delta]q&=-g^{ij} q(\omega,\partial_i\phi)\partial_j\phi_3=-g^{ij}q_{\mu\nu}\partial_i\phi_\mu\phi_\nu\partial_j\phi_3=-q_{a3}\omega_a.\label{B1Derq}
\end{align}

 Next, we repeat the same analysis for the second component of $B$. First we use $\tilde h^{\R^2}\equiv 0$ for $\tilde g=\delta$, $H[0,\delta]=2$ and Equation \eqref{HlinearizationPHI} to get 
\begin{align}
    D_1 B_2[0,\delta]\varphi&=\frac{\partial}{\partial\eta}D_1 H[0,\delta]\varphi=-\frac{\partial}{\partial\eta}(\Delta+2)\varphi,\label{B2VariationPHI}\\
    D_2 B_2[0,\delta]q&=\frac{\partial}{\partial\eta}\left(D_2 H[0,\delta]q\right)+H[0,\delta]\tilde \nu[0,\delta]^i\tilde \nu[0,\delta]^j D_2 \tilde h^{\R^2}_{ij}[\delta]q.\label{B2VariationGeneralMETVersion0}
\end{align}
Denoting the interior normal (that is pointing into $\R^3_+$) along $\R^2\times 0$ by $\tilde\nu_{\R^2}$ we have to consider
$$\tilde h^{\R^2}_{ij}=\tilde g(\nabla^{\tilde g}_{e_i}e_j, \tilde\nu_{\R^2})=\tilde g_{\alpha\beta} \Gamma^\alpha_{\ ij}\tilde\nu^\beta$$
where $\Gamma^\alpha_{\ ij}$ denote the Christoffel Symbols with respect to $\tilde g$. As $\Gamma^\alpha_{\ ij}\equiv 0$ and $\tilde\nu^{\R^2}=e_3$ for $\tilde g=\delta$ we have
$$D_2 \tilde h^{\R^2}_{ij}[\delta]q=\frac12(\partial_i q_{j3}+\partial_j q_{i3}-\partial_3 q_{ij}).$$
Inserting this and Equation \eqref{HlinearizationGeneralMET} into Equation \eqref{B2VariationGeneralMETVersion0} and using $\tilde\nu[0,\delta]=-\omega$ as well as $H[0,\delta]=2$ gives
\begin{align} 
D_2 B_2[0,\delta]q=&\frac{\partial}{\partial\eta}\left(
\frac12\tr_{\Sp^2}\nabla_\omega q-\tr_{\Sp^2}\nabla_\cdot q(\omega,\cdot)+q(\omega,\omega)-\tr_{\Sp^2} q
\right)\nonumber\\
&+\omega_i\omega_j(\partial_i q_{j3}+\partial_j q_{i3}-\partial_3 q_{ij})\label{B2VariationGeneralMET}.
\end{align}

\subsection{Formulas for the Concrete Background Metric}
In this subsection, we evaluate the formulas from Subsection \ref{Formulas} for the concrete choice 
\begin{equation}\label{MetricDerivativeFormula}
q:=\frac{\partial \tilde g^{a,\lambda}}{\partial\lambda}\bigg|_{\lambda=0}=\begin{bmatrix}
 & & \kappa_1 x_1\\
 & & \kappa_2 x_2\\
\kappa_1 x_1 & \kappa_2 x_2 & 0
\end{bmatrix}
\end{equation}
which we have already encountered in Equation \eqref{metricderivativeformula}.
\paragraph{Area, Barycenter and Volume}\ \\
Observing Equation \eqref{MetricDerivativeFormula} we see that $\operatorname{tr}_{\R^3}q=0$. So, Equation \eqref{AreaVariationPHI} implies
\begin{align}
D_2 A[0,\delta]q=&\frac12\int_{\Sp^2_+}\operatorname{tr}_{\R^3} q-q(\omega,\omega)d\mu_{\Sp^2}\nonumber\\
=&-\frac12\int_{\Sp^2_+} 2(\kappa_1\omega_1^2+\kappa_2\omega_2^2)\omega_3d\mu_{\Sp^2}\nonumber\\
=&-(\kappa_1+\kappa_2)\int_{\Sp^2_+}\omega_1^2\omega_3 d\mu_{\Sp^2}\nonumber\\
=&-\frac\pi4 H^S(a).\label{D2AEndResult}
\end{align}
Again using that $\tr_{\R^3}q=0$ we can use Equation \eqref{VolumeVariationGeneralMET} to deduce
\begin{equation}\label{VolumeVariationConcreteMET}
D_2 V[0,\delta]q=\frac12\int_{B_1^+(0)}\tr_{\R^3} qd^3 x=0.
\end{equation}
Finally, we refer to Lemma 6.1 in \cite{Metsch} for the following result: 
\begin{equation}\label{BarycenterVariationConcreteMET}
D_2 C^i[0,\delta] q=0\hspace{.5cm}\textrm{for $i=1,2$}
\end{equation}

\paragraph{Mean Curvature and Willmore Operator}\ \\
Following the proof of Lemma 9 in \cite{AK} we get $\tr_{\Sp^2}\nabla_\omega q =\tr_{\Sp^2}q=-q(\omega,\omega)$ and $\tr_{\Sp^2} \nabla_\cdot q(\omega,\cdot)=\omega_3 H^S(a)-q(\omega,\omega)$. Inserting into Equation \eqref{HlinearizationGeneralMET} gives
\begin{equation}\label{HlinearizationMET} 
D_2 H[0,\delta]q
=\frac52 q(\omega,\omega)-\omega_3 H^S(a)=5 (\kappa_1\omega_1^2+\kappa_2\omega_2^2)\omega_3 -\omega_3 H^ S(a).
\end{equation}
Note that $\tilde g^{a,\lambda}$ is a scaled pullback of the euclidean metric and hence flat. Thus $\operatorname{Ric}_{\tilde g^{a,\lambda}}\equiv 0$. Combining this observation with Equations \eqref{DWq} and \eqref{HlinearizationMET} implies
\begin{align}
D_2 W[0,\delta]q&=\frac12\Delta\left(5 (\kappa_1\omega_1^2+\kappa_2\omega_2^2)\omega_3-H^S(a)\omega_3\right).\label{DWtildegprime}
\end{align}

\paragraph{Boundary Operator}\ \\
Inserting $q$ from Equation \eqref{MetricDerivativeFormula} into Equation \eqref{B1Derq} gives
\begin{equation} 
D_2 B_1[0,\delta]q=-q_{13}\omega_1-q_{23}\omega_2=-(\kappa_1\omega_1^2+\kappa_2\omega_2^2)\label{B1VariationMET}.
\end{equation}
 Note that $\partial_i q_{j3}=\delta_{ij}\kappa_i$ and that $\partial_\eta\omega_3=1$. Inserting into Equation \eqref{B2VariationGeneralMET} gives 
\begin{align}
D_2 B_2[0,\delta]\tilde g'(0)
=&
\frac{\partial}{\partial\eta}\left(
5(\kappa_1\omega_1^2+\kappa_2\omega_2^2)\omega_3-H^S(a)\omega_3
\right)
+2(\kappa_1\omega_1^2+\kappa_2\omega_2^2)\nonumber\\
=&
7 (\kappa_1\omega_1^2+\kappa_2\omega_2^2)-H^ S(a).\label{B2VariationMETEndResult}
\end{align}

\subsection{Explicit Solution - The CMC Case}\label{CMCExplicitSolution}
We use the notation from Subsection \ref{CMCOutline}. By definition, the unique solution $u(\lambda)\in C^{2,\alpha}(\Sp^2_+)$ and the Lagrange parameters $\alpha(\lambda)$ and $\beta_i(\lambda)$ satisfy the following problem: 
\begin{align}
&H[u(\lambda),\tilde g(\lambda)]=\alpha(\lambda)+\beta_i(\lambda)\nabla C^i[u(\lambda),\tilde g(\lambda)],\label{CMCPDE}\\
&B_1[u(\lambda),\tilde g(\lambda)]=0,\label{CMCBC}\\
&V[u(\lambda),\tilde g(\lambda)]=\frac{2\pi}3,\label{CMCVolume}\\
&C^i[u(\lambda),\tilde g(\lambda)]=0\label{CMCBarycenter}\hspace{.3cm}\textrm{for }i=1,2.
\end{align}
We derive a linear PDE for $u'(0)$ by differentiating these Equations at $\lambda=0$. First, we differentiate Equations \eqref{CMCVolume} and \eqref{CMCBarycenter}. To do so, we use Equations \eqref{BarycenterVariationPHI}, \eqref{VolumeVariationPHI}, \eqref{VolumeVariationConcreteMET} and \eqref{BarycenterVariationConcreteMET} to get
\begin{align}
&\int_{\Sp^2_+}u'(0) d\mu_{\Sp^2}=D_1 V[0,\delta] u'(0)=-D_2 V[0,\delta]\tilde g'(0)=0,\label{LinearizedCMCVol}\\
&\frac3{2\pi}\int_{\Sp^2_+}u'(0)\omega_i d\mu_{\Sp^2}=D_1 C^i[0,\delta]u'(0)=-D_2 C^i[0,\delta]\tilde g'(0)=0\label{LinearizedCMCBarycenter}.
\end{align}
As we need it elsewhere, we note that in particular, Equations \eqref{AreaVariationPHI} and \eqref{LinearizedCMCVol} also imply
\begin{equation}\label{CMCAreaLinearizedPHI}
D_1 A[0,\delta] u'(0)=2\int_{\Sp^2_+} u'(0) d\mu_{\Sp^2}=0.
\end{equation}
Next, we differentiate Equation \eqref{CMCBC}. To do so, we use Equations \eqref{B1Derphi} and \eqref{B1VariationMET} to get
\begin{equation}\label{LinearizedCMCBoundaryCondition}
\frac{\partial u'(0)}{\partial\eta}=-D_1 B_1[0,\delta] u'(0)=D_2 B_1[0,\delta]\tilde g'(0)=-\kappa_1\omega_1^2-\kappa_2\omega_2^2.
\end{equation}
Finally, we differentiate Equation \eqref{CMCPDE}. To do so, we need Equations \eqref{HlinearizationPHI} and \eqref{HlinearizationMET}. Additionally, we use that $\beta_i(0)=0$ and the formula for $\nabla C^i[0,\delta]$ discussed before Equation \eqref{BarycenterVariationPHI} to deduce 
\begin{align}
(\Delta+2)u'(0)=&-D_1 H[0,\delta] u'(0)\nonumber\\
        =&D_2H[0,\delta]\tilde g'(0)-\alpha'(0)-\beta_i'(0)\nabla C^i[0,\delta]\nonumber\\
        =&5(\kappa_1\omega_1^2+\kappa_2\omega_2^2)\omega_3-H^S(a)\omega_3-\alpha'(0)+\frac{3}{2\pi}\beta_i'(0)\omega_i\label{CMCLinearizedPDE1}.
\end{align}
To obtain a formula for $\alpha'(0)$, we integrate. Using Equations \eqref{LinearizedCMCVol} and \eqref{LinearizedCMCBoundaryCondition}, we get
\begin{align} 
&-2\pi\alpha'(0)+5(\kappa_1+\kappa_2)\int_{\Sp^2_+}\omega_1^2\omega_3 d\mu_{\Sp^2}-H^S(a)\int_{\Sp^2_+}\omega_3 d\mu_{\Sp^2}\nonumber\\
=&\int_{\Sp^2_+}\Delta u'(0) d\mu_{\Sp^2}+2\int_{\Sp^2_+} u'(0)d\mu_{\Sp^2}\nonumber\\
=&-\int_{\partial\Sp^2_+}\frac{\partial u'(0)}{\partial\eta}dS\nonumber\\
=&(\kappa_1+\kappa_2)\int_{\partial\Sp^2_+}\omega_1^2 dS\nonumber\\
=&\pi H^S(a).\label{CMCalphacomputation}
\end{align}
By e.g. using spherical coordinates, it is easy to work out
$$\int_{\Sp^2_+}\omega_1^2\omega_3 d\mu_{\Sp^2}=\frac\pi4
\hspace{.5cm}\textrm{and}\hspace{.5cm}
\int_{\Sp^2_+}\omega_3 d\mu_{\Sp^2}=\pi.
$$
Inserting into Equation \eqref{CMCalphacomputation} implies 
\begin{equation}\label{CMCalphaResult}
-2\pi\alpha'(0)+\frac{5\pi}4 H^S(a)-\pi H^S(a)=\pi H^S(a)
\hspace{.5cm}\textrm{and so}\hspace{.5cm}
\alpha'(0)=-\frac38H^S(a).
\end{equation}
Multiplying Equation \eqref{CMCLinearizedPDE1} with $\omega_1$, integrating and using Gauß's theorem gives 
\begin{align*}
&\frac3{2\pi}\beta_1'(0)\int_{\Sp^2_+}\omega_1^2 d\mu_{\Sp^2}\\
=&2\int_{\Sp^2_+}u'(0)\omega_1 d\mu_{\Sp^2}+\int_{\Sp^2_+}\Delta u'(0)\omega_1 d\mu_{\Sp^2}\\
=&2\int_{\Sp^2_+}u'(0)\omega_1 d\mu_{\Sp^2}-\int_{\partial\Sp^2_+}\frac{\partial u'(0)}{\partial\eta}\omega_1-u'(0)\frac{\partial\omega_1}{\partial\eta} dS+\int_{\Sp^2_+} u'(0)\Delta \omega_1 d\mu_{\Sp^2}.
\end{align*}
Using $\partial_\eta\omega_1=0$, $\Delta\omega_1=-2\omega_1$ and Equation \eqref{LinearizedCMCBoundaryCondition} shows $\beta_1'(0)=0$ and similarly $\beta_2'(0)=0$. Inserting this and Equation \eqref{CMCalphaResult} into Equation \eqref{CMCLinearizedPDE1} as well as recalling Equations \eqref{LinearizedCMCBarycenter} and \eqref{LinearizedCMCBoundaryCondition} shows that $u'(0)$ is a solution to 
\begin{equation}\label{CMCFinalPDE}
\left\{
\begin{aligned}
    &(\Delta+2)u'(0)=5(\kappa_1\omega_1^2+\kappa_2\omega_2^2)\omega_3-H^S(a)\omega_3+\frac38 H^S(a),\\
    &\frac{\partial u'(0)}{\partial\eta}=-\kappa_1\omega_1^2-\kappa_2\omega_2^2,\\
    &\int_{\Sp^2_+}u'(0)\omega_id\mu_{\Sp^2}=0\hspace{.3cm}\textrm{for }i=1,2.
\end{aligned}
\right.
\end{equation}
It is easy to see that this problem admits only one solution; hence, $u'(0)$ is completely determined by \eqref{CMCFinalPDE}. Using a machine computation it can therefore be checked that
\begin{equation}\label{ExplicitSolutionCMCCase}
u'(0)=\frac{\kappa_1+\kappa_2}4\left(\frac34-\omega_3\right)
+\frac{\kappa_1-\kappa_2}4 \left(\frac{\omega_1^2-\omega_2^2}3\frac{2-3\omega_3+\omega_3^3}{(1-\omega_3^2)^2}\right)-\frac12(\kappa_1\omega_1^2+\kappa_2\omega_2^2)\omega_3.
\end{equation}
\paragraph{Derivation of $u'(0)$}\ \\
We outline how the formula in Equation \eqref{ExplicitSolutionCMCCase} was derived. First, we define the function 
$$f(\omega):=(\kappa_1\omega_1^2+\kappa_2\omega_2^2)\omega_3 \hspace{.5cm}\textrm{which satisfies}\hspace{.5cm}\frac12(\Delta+2) f=H^S(a)\omega_3- 5(\kappa_1\omega_1^2+\kappa_2\omega_2^2)\omega_3.$$
This shows that $v:= u'(0)+\frac12 f$ satisfies the problem 
$$\left\{\begin{aligned}
&(\Delta+2) v=\frac38H^S(a),\\
&\frac{\partial v}{\partial\eta}=-\frac12 (\kappa_1\omega_1^2+\kappa_2\omega_2^2),\\
&\int_{\Sp^2_+}\omega_i vd\mu_{\Sp^2}=0\hspace{.5cm}\textrm{for $i=1,2$}.   
\end{aligned}\right.$$
We separate this problem into two sub-problems by writing
\begin{equation}\label{CMCvAnsatz}
v=\frac{\kappa_1+\kappa_2}4 v_1+\frac{\kappa_1-\kappa_2}4v_2
\end{equation}
where $v_1$ and $v_2$ are solutions of the following problems:\\
\begin{minipage}[t]{.5\textwidth}
$$\left\{\begin{aligned}
&(\Delta+2) v_1=\frac32,\\
&\frac{\partial v_1}{\partial\eta}=-1,\\
&\int_{\Sp^2_+}\omega_i v_1d\mu_{\Sp^2}=0\hspace{.5cm}\textrm{for $i=1,2$}.   
\end{aligned}\right.$$
\end{minipage}
\begin{minipage}[t]{.5\textwidth}
$$\left\{\begin{aligned}
&(\Delta+2) v_2=0,\\
&\frac{\partial v_2}{\partial\eta}=-(\omega_1^2-\omega_2^2),\\
&\int_{\Sp^2_+}\omega_i v_2d\mu_{\Sp^2}=0\hspace{.5cm}\textrm{for $i=1,2$}.   
\end{aligned}\right.$$
\end{minipage}\ \\
These problems can be solved by making the following two ansatzes:
$$v_1(\omega)=g_1(\omega_3)
\hspace{.5cm}\textrm{and}\hspace{.5cm}v_2(\omega)=(\omega_1^2-\omega_2^2)g_2(\omega_3)$$
These ansatzes immediately satisfy the integral conditions. Additionally, working in spherical coordinates, they reduce the PDEs to two ODEs which can then be solved explicitly. For $v_1$, this is obvious. For $v_2$, we remark that 
$$\Delta v_2=(\omega_1^2-\omega_2^2)(g_2''+5\cot\theta g_2'-6g_2)
\hspace{.5cm}\textrm{and}\hspace{.5cm}
\frac{\partial v_2}{\partial\eta}=-(\omega_1^2-\omega_2^2) g_2'(\frac\pi2).$$
Note that the ODEs for $g_1$ and $g_2$ are of second order but that we only have one boundary condition, so the problems seem to be underdetermined. However, in both cases, one of the homogeneous solutions is singular and we get unique solutions by demanding smoothness of the solutions. We demonstrate this for $v_1$. The resulting ODE is
$$\frac1{\sin\theta}\frac d{d\theta}\left(\sin\theta g_1'(\theta)\right)+2g_1(\theta)=\frac32.$$
This Equation has the general solution 
$$g_1(\theta)=\frac34+c_1\cos\theta+c_2\left(\operatorname{artanh}(\cos\theta)-1\right).$$
$g_1$ is singular at $\theta=0$ unless $c_2=0$. We can then fix $c_1$ by demanding
$$-1=\frac{\partial v_1}{\partial\eta}=-\frac\partial{\partial\theta}\bigg|_{\theta=\frac\pi2}g_1(\theta)=c_1.$$
Recalling that in spherical coordinates $\cos\theta=\omega_3$, we deduce that
\begin{equation}\label{CMCv1solution}
v_1(\omega)=-\omega_3+\frac34.
\end{equation}
Repeating the same analysis for $v_2$ shows that
\begin{equation}\label{CMCv2solution}
v_2(\omega)=\frac{\omega_1^2-\omega_2^2}3\frac{2-3\omega_3+\omega_3^3}{(1-\omega_3^2)^2}.
\end{equation}
We obtain Equation \eqref{ExplicitSolutionCMCCase} by recalling the definitions of $v$ and $f$ and inserting Equations \eqref{CMCvAnsatz}, \eqref{CMCv1solution} and \eqref{CMCv2solution}.

\subsection{Explicit solution - Willmore Case}\label{explicitsolutionapp}
We use the notation from Subsection \ref{Outline}. By definition $u(\lambda)$ is the unique solution to the problem 
\begin{align}
   W[u(\lambda),\tilde g(\lambda)]&=\alpha(\lambda) H[u(\lambda),\tilde g(\lambda)]+\beta_i(\lambda) \nabla C^i[u(\lambda),\tilde g(\lambda)],\label{NonlinearWIllmore1}\\
    B[u(\lambda),\tilde g(\lambda)]&=0,\label{NonlinearWIllmore2}\\
    A[u(\lambda),\tilde g(\lambda)]&=2\pi,\label{NonlinearWIllmore3}\\
    C^i[u(\lambda),\tilde g(\lambda)]&=0\hspace{.3cm}\textrm{for }i=1,2.\label{NonlinearWIllmore4}
\end{align}
As we did in Subsection \ref{CMCExplicitSolution}, we derive a linear PDE for $u'(0)$ by differentiating these Equations at $\lambda=0$. First we differentiate Equations \eqref{NonlinearWIllmore3} and \eqref{NonlinearWIllmore4}. Using Equations \eqref{AreaVariationPHI}, \eqref{BarycenterVariationPHI}, \eqref{D2AEndResult} and \eqref{BarycenterVariationConcreteMET} we derive
\begin{align}
   & 2\int_{\Sp^2_+} u'(0) d\mu_{\Sp^2}=D_1 A[0,\delta]u'(0)=-D_2 A[0,\delta]\tilde g'(0)=\frac\pi4 H^S(a),\label{WillmorelinearizedArea}\\
    &\frac3{2\pi} \int_{\Sp^2_+}u'(0)\omega_i d\mu_{\Sp^2}=D_1 C^i[0,\delta]u'(0)=-D_2 C[0,\delta]\tilde g'(0)=0\label{WillmorelinearizedBarycenter}.
\end{align}
Next, we differentiate the boundary conditions. Following the same reasoning that lead to Equation \eqref{LinearizedCMCBoundaryCondition}, we  first differentiate the first component of Equation \eqref{NonlinearWIllmore2} to derive
\begin{equation}\label{WillmoreB1Linearized}
\frac{\partial u'(0)}{\partial\eta}=-(\kappa_1\omega_1^2+\kappa_2\omega_2^2).
\end{equation}
Next, we differentiate the second component of Equation \eqref{NonlinearWIllmore2}. Using Equations \eqref{B2VariationPHI} and \eqref{B2VariationMETEndResult}, we obtain
\begin{align}
\frac{\partial}{\partial\eta}(\Delta+2)u'(0)
=&
-D_1 B_2[0,\delta]u'(0)\nonumber\\
=&
D_2 B_2[0,\delta]\tilde g'(0)\nonumber\\
=&
7(\kappa_1\omega_1^2+\kappa_2\omega_2^2)-H^ S(a).\label{secondBCFO}
\end{align}
Finally, we differentiate Equation \eqref{NonlinearWIllmore1}. Using that $\alpha(0)=\beta_1(0)=\beta_2(0)=0$ as well as Equations \eqref{DWphi} and \eqref{DWtildegprime}, we obtain
\begin{align*}
\frac12\Delta(\Delta+2) u'(0)=&-D_1 W[0,\delta] u'(0)\\
=& D_2 W[0,\delta]\tilde g'(0)-\alpha'(0) H[0,\delta]-\beta_i'(0) \nabla C^i[0,\delta]\\
=& \frac12\Delta\left(5(\kappa_1\omega_1^2+\kappa_2\omega_2^2)\omega_3-\omega_3H^S(a)\right)-\alpha'(0) H[0,\delta]-\beta_i'(0) \nabla C^i[0,\delta].
\end{align*}
We multiply by 2, insert $H[0,\delta]=2$ and use the formula for $\nabla C^i[0,\delta]$ that we discussed just before Equation \eqref{BarycenterVariationPHI} to get
\begin{equation}\label{LienarizedPDEWIllmore01}
\Delta(\Delta+2) u'(0)=\Delta\left(5(\kappa_1\omega_1^2+\kappa_2\omega_2^2)\omega_3-\omega_3H^S(a)\right)-4\alpha'(0)+\frac3\pi\beta_i'(0)\omega_i.
\end{equation}
Next, we multiply Equation \eqref{LienarizedPDEWIllmore01}  by $\omega_1$ and integrate. Inserting Equations \eqref{WillmoreB1Linearized} and \eqref{secondBCFO} and using $\Delta\omega_1=-2\omega_1$ shows
\begin{align*} 
\frac3\pi\beta_1'(0)\int_{\Sp^2_+}\omega_1^2 d\mu_{\Sp^2}=&\int_{\Sp^2_+}\Delta(\Delta+2) u'(0) \omega_1 d\mu_{\Sp^2}\\
=& 2\int_{\Sp^2_+}\Delta u'(0) \omega_1 d\mu_{\Sp^2}+\int_{\Sp^2_+}\Delta^2 u'(0)\omega_1 d\mu_{\Sp^2}\\
=&-\int_{\partial\Sp^2_+}\frac{\partial \Delta u''(0)}{\partial\eta}\omega_1dS\\
=&0.
\end{align*}
Hence $\beta_1'(0)=0$ and similarly $\beta_2'(0)=0$. To work out $\alpha'(0)$ we integrate Equation \eqref{LienarizedPDEWIllmore01} and insert \eqref{secondBCFO} to find 
\begin{align*}
-8\pi\alpha'(0)=&\int_{\Sp^2_+}\Delta(\Delta+2) u'(0) d\mu_{\Sp^2}-\int_{\Sp^2_+}\Delta\left(5(\kappa_1\omega_1^2+\kappa_2\omega_2^2)\omega_3-\omega_3H^S(a)\right)d\mu_{\Sp^2}\\
=&-\int_{\partial \Sp^2_+}\frac\partial{\partial\eta}(\Delta+2) u'(0) dS+\int_{\partial \Sp^2_+}\frac\partial{\partial\eta}\left(5(\kappa_1\omega_1^2+\kappa_2\omega_2^2)\omega_3-\omega_3H^S(a)\right)dS\\
=&-\int_{\partial \Sp^2_+}\frac\partial{\partial\eta}(\Delta+2) u'(0) dS+\int_{\partial \Sp^2_+}5(\kappa_1\omega_1^2+\kappa_2\omega_2^2)-H^S(a)dS\\
=&-2\int_{\partial \Sp^2_+}\kappa_1\omega_1^2+\kappa_2\omega_2^2 dS\\
=&-2H^S(a)\int_{\partial\Sp^2_+}\omega_1^2 dS\\
=&-2\pi H^S(a).
\end{align*}
Inserting this and $\beta_1'(0)=\beta_2'(0)=0$ into Equation \eqref{LienarizedPDEWIllmore01}, we get
\begin{equation}\label{WillmoreLinearizedPDE}
\Delta(\Delta+2) u'(0)=\Delta\left(5(\kappa_1\omega_1^2+\kappa_2\omega_2^2)\omega_3-\omega_3H^S(a)\right)-H^S(a).
\end{equation}
Collecting Equations \eqref{WillmorelinearizedArea}, \eqref{WillmorelinearizedBarycenter}, \eqref{WillmoreB1Linearized}, \eqref{secondBCFO} and \eqref{WillmoreLinearizedPDE} shows that $u'(0)$ is a solution to the following problem:
\begin{equation}\label{WillmoreLienarizedPDEFinal}\left\{\begin{aligned}
&\Delta(\Delta+2) u'(0)=\Delta\left(5(\kappa_1\omega_1^2+\kappa_2\omega_2^2)\omega_3-\omega_3H^S(a)\right)-H^S(a),\\
&\frac{\partial u'(0)}{\partial\eta}=-(\kappa_1\omega_1^2+\kappa_2\omega_2^2),\\
&\frac{\partial}{\partial\eta}(\Delta+2)u'(0)=7 (\kappa_1\omega_1^2+\kappa_2\omega_2^2)-H^S(a),\\
&\int_{\Sp^2_+} u'(0) d\mu_{\Sp^2}=\frac\pi8 H^S(a),\\
&\int_{\Sp^2_+} u'(0)\omega_id\mu_{\Sp^2}=0\hspace{.3cm}\textrm{for $i=1,2$}.
\end{aligned}\right.
\end{equation}
It is easy to see that this problem admits only one solution; hence, $u'(0)$ is completely determined by \eqref{WillmoreLienarizedPDEFinal}. Using a machine computation it can therefore be checked that
\begin{equation}\label{WillmoreLinearizedSolution}
u'(0)=\frac{\kappa_1-\kappa_2}4\frac{\omega_1^2-\omega_2^2}{1+\omega_3}+(\kappa_1+\kappa_2)\left[1-\ln(2)+\frac{\ln(1+\omega_3)}2-\frac34\omega_3\right]-\frac{\kappa_1\omega_1^2+\kappa_2\omega_2^2}2\omega_3.
\end{equation}
 \paragraph{Derivation of $u'(0)$}\ \\
We find $u'(0)$ using the same method that we employed in the CMC case. First we introduce $f(\omega):=(\kappa_1\omega_1^2+\kappa_2\omega_2^2)\omega_3$ and define $v:=u'(0)+\frac12 f$. $v$ satisfies the problem
$$\left\{ \begin{aligned}
    \Delta(\Delta+2)v&=-H^S(a),\\
    \frac{\partial v}{\partial\eta}&=-\frac12(\kappa_1\omega_1^2+\kappa_2\omega_2^2),\\
    \frac{\partial(\Delta+2)v}{\partial\eta}&=2(\kappa_1\omega_1^2+\kappa_2\omega_2^2),\\
    \int_{\Sp^2_+} v d\mu_{\Sp^2}&=\frac\pi4 H^S(a),\ \int_{\Sp^2_+} v\omega_id\mu_{\Sp^2} =0.
\end{aligned}\right.$$
To solve this PDE, we make the ansatz
\begin{equation}\label{WillmorevAnsatz}
v=(\kappa_1+\kappa_2)v_1+\frac{\kappa_1-\kappa_2}4 v_2
\end{equation}
where $v_1$ and $v_2$ are solutions to the following problems:
$$
\left\{ \begin{aligned}
    \Delta(\Delta+2)v_1&=-1,\\
    \frac{\partial v_1}{\partial\eta}&=-\frac14,\\
    \frac{\partial(\Delta+2)v_1}{\partial\eta}&=1, \\
    \int_{\Sp^2_+}v_1d\mu_{\Sp^2}&=\frac\pi4,\ \int_{\Sp^2_+}v_1\omega_id\mu_{\Sp^2}=0.
\end{aligned}\right.
\hspace{.5cm}
\textrm{and}
\hspace{.5cm}
\left\{ \begin{aligned}
    \Delta(\Delta+2)v_2&=0,\\
    \frac{\partial v_2}{\partial\eta}&=-(\omega_1^2-\omega_2^2),\\
    \frac{\partial(\Delta+2)v_2}{\partial\eta}&=4( \omega_1^2-\omega_2^2),\\
    \int_{\Sp^2_+}v_2d\mu_{\Sp^2}&=0,\     \int_{\Sp^2_+}v_2\omega_id\mu_{\Sp^2}=0.
\end{aligned}\right.
$$
Using the same strategy that we used to solve the analog equations in the CMC case, we get the following solutions for $v_1$ and $v_2$:
\begin{equation}\label{v1v2SolutionsWIllmoreCase}
v_1(\omega)=1-\ln(2)+\frac12\ln(1+\omega_3)-\frac34\omega_3
\hspace{.5cm}\textrm{and}\hspace{.5cm}
v_2(\omega)=\frac{\omega_1^2-\omega_2^2}{1+\omega_3}
\end{equation}
Recalling the definitions of $v$ and $f$, the ansatz \eqref{WillmorevAnsatz} and the formulas from Equation \eqref{v1v2SolutionsWIllmoreCase}, we get Equation \eqref{WillmoreLinearizedSolution}.

\subsection{Second derivatives - CMC Case}\label{secondDerAppCMC}
We use the notation from Subsection \ref{CMCOutline}. The first step is to differentiate the volume constraint $V[u(\lambda),\tilde g(\lambda)]=2\pi/3$ twice with respect to $\lambda$ and evaluate at $\lambda=0$. This gives
\begin{align}
0=&D_1 V[0,\delta] u''(0)+D_2 V[0,\delta]\tilde g''(0)\nonumber\\
&+D_1^2 V[0,\delta] (u'(0), u'(0))+D_2^2V[0,\delta](\tilde g'(0),\tilde g'(0))+2D_{12}^2V[0,\delta](u'(0),\tilde g'(0)).\label{VolumeConstraintSecondDERIVATIVE}
\end{align}
Recalling that we denote the interior domain that is enclosed by the graph of $(1+\epsilon u'(0))\omega$ by $\Omega_{\epsilon u'(0)}$ and using that $\tr_{\R^3}\tilde g'(0)=0$ (see Equation \eqref{metricderivativeformula}) we compute 
\begin{align} 
D_{12}^2 V[0,\delta](u'(0),\tilde g'(0))=&\frac{d}{d\epsilon}\bigg|_{\epsilon=0}\frac{d}{d\alpha}\bigg|_{\alpha=0}V[\epsilon u'(0), \delta+\alpha\tilde g'(0)]\nonumber\\
=&\frac {d}{d\epsilon}\bigg|_{\epsilon=0}\left[\int_{\Omega_{\epsilon u'(0)}}\frac{d}{d\alpha}\bigg|_{\alpha=0}\sqrt{\det(\delta+\alpha\tilde g'(0))}d^3x\right]\nonumber\\
=&\frac {d}{d\epsilon}\bigg|_{\epsilon=0}\left[\int_{\Omega_{\epsilon u'(0)}}\frac12\tr_{\R^3}\tilde g'(0)\right]d^3x\nonumber\\
=&0.\label{MixedVolumeDerivativeVANISH}
\end{align}
Next, we need an explicit formula for $\tilde g''(0)$. This formula is derived in Appendix A in \cite{Metsch}. Using our concrete choice of basis in which $h_{ij}=\delta_{ij}\kappa_i$ at the point $a$, the formula simplifies to
\begin{equation}\label{tildegdoubleprime}
\tilde g''(0)=\begin{bmatrix}
2\kappa_1^2x_1^2 & 2\kappa_1\kappa_2x_1x_2 & \partial_1 h_{ab}x_ax_b\\
2\kappa_1\kappa_2x_1x_2 & 2\kappa_2^2x_2^2 & \partial_2 h_{ab}x_ax_b\\
\partial_1 h_{ab}x_ax_b& \partial_2 h_{ab}x_ax_b &0
\end{bmatrix}.
\end{equation}
Here $\partial_i h_{ab}$ denotes the derivative of the second fundamental form with respect to the base point. Observing Equation \eqref{tildegdoubleprime}, we get $\tr_{\R^3}\tilde g''(0)=2(\kappa_1^2x_1^2+\kappa_2^2x_2^2)$ and hence 
\begin{equation}\label{VariationMETg2prime}
D_2 V[0,\delta]\tilde g''(0)=\frac12\int_{B_1^+(0)}2(\kappa_1^2x_1^2+\kappa_2^2x_2^2)d^3x=\int_{B_1^+(0)}\kappa_1^2x_1^2+\kappa_2^2x_2^2 d^3x.
\end{equation}
For the moment let us abbreviate $M_\epsilon:=\delta+\epsilon\tilde g'(0)$, using $\tr_{\R^3}(M_0^{-1}\tilde g'(0))=\tr_{\R^3}\tilde g'(0)=0$ we compute 
\begin{align}
D_2^2 V[0,\delta](\tilde g'(0),\tilde g'(0))=&\frac{d^2}{d\epsilon^2}\bigg|_{\epsilon=0}V[0, M_\epsilon]\nonumber\\
=&\frac{d^2}{d\epsilon^2}\bigg|_{\epsilon=0}\int_{B_1^+(0)}\sqrt{\det(M_\epsilon)}d^3x\nonumber\\
=&\frac d{d\epsilon}\bigg|_{\epsilon=0}\int_{B_1^+(0)}\frac1{2}\sqrt{\det(M_\epsilon)}\tr_{\R^3}(M_\epsilon^{-1}\tilde g'(0))d^3x\nonumber\\
=&\frac12\int_{B_1^+(0)}\tr_{\R^3}\left[\frac{\partial M_\epsilon^{-1}}{\partial\epsilon}\bigg|_{\epsilon=0}\tilde g'(0)\right]d^3x\nonumber\\
=&-\frac12 \int_{B_1^+(0)}\tr_{\R^3}(\tilde g'(0)^2)d^3x\label{SecondVlumeVariationMET}.
\end{align}
Observing Equation \eqref{metricderivativeformula} and denoting irrelevant matrix indices by $*$, we get
$$\tr_{\R^3}(\tilde g'(0)^2)=\tr_{\R^3}\begin{bmatrix}
\kappa_1^2x_1^2 & * & *\\
* & \kappa_2^2 x_2^2 & *\\
* & * & \kappa_1^2 x_1^2+\kappa_2^2 x_2^2
\end{bmatrix}
=2(\kappa_1^2 x_1^2+\kappa_2^2 x_2^2).$$
Inserting into Equation \eqref{SecondVlumeVariationMET} and comparing with Equation \eqref{VariationMETg2prime} shows
\begin{equation}\label{VOlumeSecondMETDerivativeCancelation}
D_2V[0,\delta]\tilde g''(0)+D_2^2V[0,\delta](\tilde g'(0),\tilde g'(0))=0.
\end{equation}
Inserting Equations \eqref{VolumeVariationPHI}, \eqref{MixedVolumeDerivativeVANISH} and \eqref{VOlumeSecondMETDerivativeCancelation} into Equation \eqref{VolumeConstraintSecondDERIVATIVE}, we get
$$\int_{\Sp^2_+}u''(0) d\mu_{\Sp^2}=D_1 V[0,\delta] u''(0)=-D_1^2 V[0,\delta](u'(0), u'(0))=-2\int_{\Sp^2_+}u'(0)^2 d\mu_{\Sp^2}.$$
Considering Equation \eqref{AreaVariationPHI} and using the explicit formula for $u'(0)$ from Equation \eqref{ExplicitSolutionCMCCase}, we use a machine computation to get 
\begin{align*} 
D_1 A[0,\delta]u''(0)=&2\int_{\Sp^2_+}u''(0)d\mu_{\Sp^2}\\
=&-4\int_{\Sp^2_+}u'(0)^2 d\mu_{\Sp^2}\\
=&\left(\frac{113}{30240}-\frac{\ln(2)}{9}\right)\pi H^S(a)^2+\left(-\frac{229}{945}+\frac49\ln(2)\right)\pi K^S(a).
\end{align*}
This is Equation \eqref{CMC5}. To establish Equation \eqref{CMC4}, we use Equation \eqref{AreaVariationPHI} and insert the formula for $\tilde g''(0)$ from Equation \eqref{tildegdoubleprime}. This gives
\begin{equation}\label{D2AtildegDoublePrime}
D_2 A[0,\delta]\tilde g''(0)=\frac12\int_{\Sp^2_+}\tr_{\Sp^2}\tilde g''(0) d\mu_{\Sp^2}
=\frac12\int_{\Sp^2_+}\tr_{\R^3}\tilde g''(0)-\tilde g''(0)(\omega,\omega) d\mu_{\Sp^2}.
\end{equation}
Considering Equation \eqref{tildegdoubleprime}, we have
\begin{align*}
    \tr_{\R^3} \tilde g''(0)&=2(\kappa_1^2x_1^2+\kappa_2^2x_2^2),\\
    \tilde g''(0)(x,x)&=2\kappa_1^2x_1^4+2\kappa_2^2x_2^4+4\kappa_1\kappa_2x_1^2x_2^2+\textrm{odd}.
\end{align*}
Therefore:
\begin{align}
    \frac12\int_{\Sp^2_+}\tr_{\R^3}\tilde g''(0)d\mu_{\Sp^2}&=(\kappa_1^2+\kappa_2^2)\int_{\Sp^2_+}\omega_1^2d\mu_{\Sp^2}=\frac{2\pi}3H^S(a)^2-\frac{4\pi}3K^S(a)\label{gDoublePrimeINtegral1}\\
    \ \nonumber\\
%--------------------------
    \frac12\int_{\Sp^2_+}\tilde g''(0)(\omega,\omega)d\mu_{\Sp^2}&=
    (\kappa_1^2+\kappa_2^2)\int_{\Sp^2_+}\omega_1^4d\mu_{\Sp^2}+2 \kappa_1\kappa_2\int_{\Sp^2_+}\omega_1^2\omega_2^2d\mu_{\Sp^2}\\
     &=\frac{2\pi}5 H^S(a)^2-\frac{8\pi}{15}K^S(a)\label{gDoublePrimeINtegral2}
\end{align}
Inserting these formulas into Equation \eqref{D2AtildegDoublePrime} implies Equation \eqref{CMC4}.

\subsection{Second Derivatives - Willmore Case}\label{secondderivativeWillmore}
We use the notation from Subsection \ref{Outline}. First, we differentiate the first boundary condition $B_1[u(\lambda),\tilde g(\lambda)]=0$ twice with respect to $\lambda$ and evaluate at $\lambda=0$. Using Equation \eqref{B1Derphi}, we get 
\begin{align*}
0=& D_1 B_1[0,\delta] u''(0)+D_2 B_1[0,\delta]\tilde g''(0)+\frac{d^2}{d\epsilon^2}\bigg|_{\epsilon=0} B_1[\epsilon u'(0), \delta+\epsilon\tilde g'(0)]\\
=&-\frac{\partial u''(0)}{\partial\eta}+D_2 B_1[0,\delta]\tilde g''(0)+\frac{d^2}{d\epsilon^2}\bigg|_{\epsilon=0} B_1[\epsilon u'(0), \delta+\epsilon\tilde g'(0)].
\end{align*}
Combining with Equation \eqref{WillEnDeru}, we deduce 
\begin{align}
D_1 \mathcal W[0,\delta] u''(0)=&\int_{\partial \Sp^2_+} \frac{\partial u''(0)}{\partial\eta} dS\nonumber\\
=&\int_{\partial\Sp^2_+}D_2 B_1[0,\delta]\tilde g''(0)+\frac{d^2}{d\epsilon^2}\bigg|_{\epsilon=0} B_1[\epsilon u'(0), \delta+\epsilon\tilde g'(0)]dS.\label{WillmoreVariationu2PrimeV0}
\end{align}
Using Equation \eqref{B1Derq} and the formula for $\tilde g''(0)$ from Equation \eqref{tildegdoubleprime} we get $D_2 B_1[0,\delta]\tilde g''(0)=-\tilde g''(0)_{a3}\omega_a=-\partial_ a h_{ij}\omega_a\omega_i\omega_j$. As this is an odd term we get 
\begin{equation}\label{oddtermdissappeards}
\int_{\partial\Sp^2_+}D_2 B_1[0,\delta]\tilde g''(0)dS=0.
\end{equation}
Inserting into Equation \eqref{WillmoreVariationu2PrimeV0} and using the explicit formulas for $\tilde g'(0)$ from Equation \eqref{metricderivativeformula} as well as for $u'(0)$ from Equation \eqref{WillmoreLinearizedSolution}, we can use Mathematica to compute 
$$D_1 \mathcal W[0,\delta] u''(0)=\int_{\partial\Sp^2_+}\frac{d^2}{d\epsilon^2}\bigg|_{\epsilon=0} B_1[\epsilon u'(0), \delta+\epsilon\tilde g'(0)]dS=4\pi H^S(a)^2(\ln(2)-1).$$
This is Equation \eqref{4}. To establish Equation \eqref{5} let us abbreviate $q=\tilde g''(0)$. Then first $q(t\omega)=t^2 q(\omega)$ and so we may deduce $\nabla_\omega q(\omega)=2 q$. Also we use $\tr_{\Sp^2} q=\tr_{\R^3} q-q(\omega,\omega)$. Inserting into Equation \eqref{WillEnDerq} gives 
\begin{align} 
D_2 \mathcal W[0,\delta]q&=\int_{\Sp^2_+}\frac12\tr_{\Sp^2}q+q(\omega,\omega)-\tr_{\Sp^2}\nabla_\cdot q(\cdot,\omega) d\mu_{\Sp^2}\nonumber\\
&=\int_{\Sp^2_+}\frac12\tr_{\R^3}q+\frac12q(\omega,\omega)-\tr_{\R^3}\nabla_\cdot q(\cdot,\omega)+\nabla_\omega q(\omega,\omega) d\mu_{\Sp^2}\nonumber\\
&=\int_{\Sp^2_+}\frac12\tr_{\R^3}q+\frac52q(\omega,\omega)-\tr_{\R^3}\nabla_\cdot q(\cdot,\omega)d\mu_{\Sp^2}.\label{D2Wtildegdoubleprime}
\end{align}
Considering Equation \eqref{tildegdoubleprime}, we have 
\begin{align*}
    \tr_{\R^3}\nabla_\cdot q(\cdot,x)&=\sum_{i=1}^2\partial_i q(e_i, x)=\sum_{i=1}^2\sum_{\mu=1}^3\partial_i q_{i\mu}x_\mu
    =\sum_{i=1}^2\sum_{j=1}^2\partial_i q_{ij}x_j+\textrm{odd}\\
    &=4(\kappa_1^2x_1^2+\kappa_2^2x_2^2)+2\kappa_1\kappa_2(x_1^2+x_2^2)+\textrm{odd}.
\end{align*}
Using this formula, it is straightforward to check that
\begin{align}
  %--------------------------
-\int_{\Sp^2_+}\tr_{\R^3}\nabla_\cdot q(\cdot,\omega)d\mu_{\Sp^2}&=-4(\kappa_1^2+\kappa_2^2)\int_{\Sp^2_+}\omega_1^2 d\mu_{\Sp^2}-4\kappa_1\kappa_2\int_{\Sp^2_+}\omega_1^2 d\mu_{\Sp^2}\nonumber\\ 
 &=-\frac{8\pi}3 H^S(a)^2+\frac{8\pi}3 K^S(a).\label{gDoublePrimeINtegral3}
 \end{align}
Inserting Equations \eqref{gDoublePrimeINtegral1}, \eqref{gDoublePrimeINtegral2} and \eqref{gDoublePrimeINtegral3} into Equation \eqref{D2Wtildegdoubleprime} implies Equation \eqref{5}.

\setcounter{equation}{0}
\section{Mathematica Code}\label{mathematicacodeAppendix}
The supplementary file contains the following commented Mathematica notebooks. 

\begin{itemize}
    \item \textbf{Domain Construction}
        \begin{enumerate}[(1)]
            \item `Section 4.1 - Proof of Equation (4.17)'
        \end{enumerate}
    \item \textbf{CMC Computations}
        \begin{enumerate}[(1)]
            \item `Proof of Equation (5.14) - The linearized CMC Equation'
            \item `Proof of Equation (5.15)'
            \item `Proof of Equations (5.16) and (5.19)'
            \item `Proof of Equation (5.17)'
        \end{enumerate}
    \item \textbf{Willmore Computations}
        \begin{enumerate}[(1)]
            \item `Check of Equation (5.3) - The linearized Willmore Equation'
            \item `Proof of Equation (5.4)'
            \item `Proof of Equation (5.5)'
            \item `Proof of Equation (5.6)'
            \item `Proof of Equation (5.7)'
        \end{enumerate}  
\end{itemize}

\paragraph{A Note regarding the Implementation}\ \\
In the Willmore computations, we are sometimes faced with the task of implementing the second fundamental form of a surface $\phi:\Omega\subset\R^2\rightarrow(\R^3,\tilde g)$. To do this, we quickly derive a suitable expression.\\

Let $\phi:\Omega\subset\R^2\rightarrow(\R^3,\tilde g)$ parameterize a surface and $\nu(x)$ denote the interior unit normal. For $\epsilon>0$ small, 
$$\Phi:\Omega\times(-\epsilon,\epsilon)\rightarrow\R^3,\ \Phi(x,z):=\phi(x)+z\nu(x)$$
is a diffeomorphism onto its image. Denoting $z$ as $x_3$, we have by definition 
\begin{align*}
h_{ij}(x)=&\tilde g(\nabla_{\partial_i\phi}\partial_j\phi,\nu)\\
=&\tilde g(\nabla_{\partial_i\Phi}\partial_j\Phi,\nu)\bigg|_{z=0}\\
=&\tilde g(\nabla_{\partial_i\Phi}\partial_j\Phi,\partial_3\Phi)\bigg|_{z=0}\\
=&\Gamma^3_{\ ij}(x_1,x_2,0).
\end{align*}
In the coordinate system $\Phi$
\begin{align*}
&\tilde g=\begin{bmatrix}
\tilde g(\partial_i\phi(x)+z\partial_i\nu(x),\partial_j\phi(x)+z\partial_j\nu(x)) & z\tilde g(\partial_i\nu(x),\nu(x))\\
z\tilde g(\partial_i\nu(x),\nu(x)) & 1
\end{bmatrix},
\\
&\tilde g\bigg|_{z=0}=\begin{bmatrix}
\tilde g(\partial_i\phi(x),\partial_j\phi(x)) &0\\
0 & 1
\end{bmatrix}.
\end{align*}
This gives 
\begin{align}
h_{ij}(x)=\Gamma^{3}_{\ ij}(x,0)=&-\frac12 \frac{\partial}{\partial z}\bigg|_{z=0}\tilde g_{\phi(x)+z\nu(x)}(\partial_i\phi(x)+z\partial_i\nu(x), \partial_j\phi(x)+z\partial_j\nu(x))\nonumber\\
=&-\frac12\left(\nu^\mu\partial_\mu \tilde g_{\alpha\beta}\partial_i\phi(x)^\alpha\partial_j\phi(x)^\beta+\tilde g_{\alpha\beta}\partial_i\phi^\alpha\partial_j\tilde\nu+\tilde g_{\alpha\beta}\partial_j\phi^\alpha\partial_i\tilde\nu\right).\label{Paper2_SecondFFExplicitFormula}
\end{align}
In the code, we implement the matrix
\begin{equation}\label{HilfsMatrixS}
\begin{aligned}
&s_{ij}:=\tilde g_{\phi(x)+z\nu(x)}(\partial_i\phi(x)+z\partial_i\nu(x), \partial_j\phi(x)+z\partial_j\nu(x))\\
&\textrm{so that}\hspace{.5cm}
h_{ij}=-\frac12\frac\partial{\partial z}\bigg|_{z=0}s_{ij}.
\end{aligned}
\end{equation}

\section*{Acknowledgements}
The author would like to thank Ernst Kuwert for the suggestion of this interesting topic and the many helpful discussions.

\bibliographystyle{plainurl}
\bibliography{quellen}
\end{document}